\documentclass[10pt]{amsart}
\usepackage{amssymb}
\usepackage{amsmath}
\usepackage{epsfig}
\usepackage{graphics}

\font\tenmath=msbm10 \font\sevenmath=msbm7 \font\fivemath=msbm5
\newfam\mathfam \textfont\mathfam=\tenmath
\scriptfont\mathfam=\sevenmath \scriptscriptfont\mathfam=\fivemath

\def \\ { \cr }

\newcommand{\vvert}{|\!|\!|}
\newcommand{\RR}{{\mathbb R}}
\newcommand{\CC}{{\mathbb C}}
\newcommand{\NN}{{\mathbb N}}
\newcommand{\ZZ}{\mathbb Z}
\newcommand{\QQ}{{\mathbb Q}}
\newcommand{\EE}{{\mathbb E}}

\def\B{{\mathcal B}}

\def\P{{\mathcal P}}

\def\T{{\mathcal T}}

\numberwithin{equation}{section}

\newtheorem{theo}{Theorem}
\newtheorem{prop}[theo]{Proposition}
\newtheorem{coro}[theo]{Corollary}
\newtheorem{lemma}[theo]{Lemma}

\theoremstyle{remark}
\newtheorem{remark}[theo]{Remark}

\begin{document}

\parindent = 0cm

\title
{Eigenvalues of finite rank
Bratteli-Vershik dynamical systems }

\author{Xavier Bressaud}
\address{Institut de Math\'ematiques de Luminy, 163 avenue de Luminy, Case
907, 13288 Marseille Cedex 9, France.}
\email{bressaud@iml.univ-mrs.fr}

\author{Fabien Durand}
\address{Laboratoire Ami\'enois
de Math\'ematiques Fondamentales et Appliqu\'ees, CNRS-UMR 6140,
Universit\'{e} de Picardie Jules Verne, 33 rue Saint Leu, 80000
Amiens, France.} \email{fabien.durand@u-picardie.fr}

\author{Alejandro Maass}
\address{Departamento de Ingenier\'{\i}a
Matem\'atica and Centro de Modelamiento
Ma\-te\-m\'a\-ti\-co, CNRS-UMI 2807, Universidad de Chile, Avenida Blanco Encalada 2120, Santiago, Chile.}
 \email{amaass@dim.uchile.cl}


\begin{abstract}
In this article we study conditions to be a continuous or a
measurable eigenvalue of finite rank minimal Cantor systems,
that is, systems given by an ordered Bratteli diagram with a bounded number
of vertices per level. 
We prove that continuous eigenvalues always
come from the stable subspace associated to the incidence matrices  
of the Bratteli diagram and we study rationally
independent generators of the additive group of continuous
eigenvalues. Given an ergodic probability measure, we provide a
general necessary condition to be a measurable eigenvalue. Then we
consider two families of examples. A first one to illustrate that
measurable eigenvalues do not need to come from the stable
space. Finally we study Toeplitz type Cantor
minimal systems of finite rank. We recover classical results in
the continuous case and we prove measurable eigenvalues are always
rational but not necessarily continuous.
\end{abstract}

\date{February 8, 2006}
\maketitle \markboth{Xavier Bressaud, Fabien Durand, Alejandro
Maass}{Eigenvalues of finite rank
systems}

%%%%%%%%%%%%%%%%%%%%%%%%%%%%%%%%%%%%%%%%%%%%%%%%%%%%%%%%%%%%%%%%%%
\section{Introduction}
%%%%%%%%%%%%%%%%%%%%%%%%%%%%%%%%%%%%%%%%%%%%%%%%%%%%%%%%%%%%%%%%%%

The study of eigenvalues of dynamical systems has been extensively considered in ergodic theory to understand and build the Kronecker  factor and also to study the weak mixing property. 
In topological dynamics one also consider continuous eigenvalues, that is, eigenvalues associated to continuous eigenfunctions, to study topological weak mixing (at least in the minimal case). 
Since continuous eigenvalues are also eigenvalues, 
a recurrent question is to know whether they coincide. 
In general the answer is negative, since there are minimal topologically weakly mixing systems that are not weakly mixing for some invariant measure.  A positive answer to that question has been given for the class of primitive substitution systems in \cite{Ho}. 
The same question has been considered for linearly recurrent Cantor minimal systems, which contains substitution systems, in \cite{CDHM} and \cite{BDM} concluding that in general not all eigenvalues are continuous. 
Nevertheless, explicit necessary and sufficient conditions  are given to check whether a complex number is an eigenvalue, continuous or not, that allow to recover the result in \cite{Ho}.  
Those conditions only depend on the incidence matrices associated to a Bratteli-Vershik representation of the linearly recurrent minimal system and not on the partial order of the diagram. The independence of the order seems to be characteristic of linearly recurrent systems.

\medskip

In this article we consider the same question for Cantor minimal systems that admit a Bratteli-Vershik representation with the same number of vertices per level. We call them (topologically) finite rank Cantor minimal systems. The motivations are different. First  this class  contains linearly recurrent systems but it is much larger and natural (some systems in this class  appear as the symbolic representation of well studied classes of dynamical systems like interval exchange transformations \cite{GJ}). Second, the knowledge about this class of systems is very small, one of its main (and recent) properties is that they are either expansive or equicontinuous  \cite{DM}. Thus, up to odometers that are well known,  this is a huge class of symbolic minimal systems. 

\medskip

In Section 3 we study continuous eigenvalues. 
The main result states that continuous eigenvalues of Cantor minimal systems always  come from the stable space associated to the sequence of matrices of the Bratteli-Vershik representation of the system and we provide a general necessary condition to be a continuous eigenvalue.  
In Section 4 these results are used to get a bound for  the maximal number of rationally independent continuous eigenvalues of a finite rank system. 
In section 5 we give  a necessary condition to be an eigenvalue  of a finite rank minimal systems indowed with an invariant probability measure.  
Next two sections are devoted to examples. In section 6 we construct a uniquely ergodic non weakly mixing finite rank Cantor minimal system to illustrate that eigenvalues do not always come from the stable space associated to the sequence of matrices given by the Bratteli-Vershik representation of the system. In the last section we study Toeplitz type Cantor minimal systems of finite rank. The main property we deduce is that eigenvalues are always rational but not always continuous.

%%%%%%%%%%%%%%%%%%%%%%%%%%%%%%%%%%%%%%%%%%%%%%%%%%%%%%%%%%%%%%%%%%
\section{Basic definitions}
%%%%%%%%%%%%%%%%%%%%%%%%%%%%%%%%%%%%%%%%%%%%%%%%%%%%%%%%%%%%%%%%%%

\subsection{Dynamical systems and eigenvalues}
A \emph{topological dynamical system}, or just dynamical system,
is a compact Hausdorff space $X$ together with a homeomorphism $
T:X\rightarrow X$. One uses the notation \( \left( X,T\right) \).
If \( X \) is a Cantor set one says that the system is Cantor.
That is, \( X \) has a countable basis of closed and open sets and
it has no isolated points. A dynamical system is \emph{minimal} if
all orbits are dense in \( X, \) or equivalently the only non
trivial closed invariant set is \( X. \)

A complex number $\lambda$ is a {\it continuous eigenvalue} of
$(X,T)$ if there exists a continuous function $f : X\to \CC$,
$f\not = 0$, such that $f\circ T = \lambda f$; $f$ is called a
{\it continuous eigenfunction} (associated to $\lambda$). Let
$\mu$ be a $T$-invariant probability measure, i.e., $T\mu = \mu$,
defined on the Borel $\sigma$-algebra $\B(X)$ of $X$. A complex
number $\lambda$ is an {\it eigenvalue} of the dynamical system
$(X,T)$ with respect to $\mu$ if there exists $f\in L^2
(X,\B(X),\mu)$, $f\not = 0$, such that $f\circ T = \lambda f$; $f$
is called an {\it eigenfunction} (associated to $\lambda$). If the
system is ergodic, then every eigenvalue is of modulus 1, and
every eigenfunction has a constant modulus $\mu$-almost surely. Of
course continuous eigenvalues are eigenvalues.

\subsection{Bratteli-Vershik representations}
Let $(X,T)$ be a Cantor minimal system. It can be represented by
an \emph{ordered Bratteli-Vershik diagram}. For details on this
theory see \cite{HPS}. 
We give a brief outline of such
constructions emphasizing the notations used in the paper.

\subsubsection{Bratteli-Vershik diagrams}
A Bratteli-Vershik diagram is an infinite graph \( \left(
V,E\right) \) which consists of a vertex set \( V \) and an edge
set \( E \), both of which are divided into levels \( V=V_{0}\cup
V_{1}\cup \ldots  \), \( E=E_{1}\cup E_{2}\cup \ldots \) and all
levels are pairwise disjoint. The set \( V_{0} \) is a singleton
\( \{v_{0}\} \), and for \( k\geq 1 \), \( E_{k} \) is the set of
edges joining vertices in \( V_{k-1} \) to vertices in \( V_{k}
\). It is also required that every vertex in \( V_{k} \) is the
``end-point'' of some edge in \( E_{k} \) for \( k\geq 1 \), and
an ``initial-point'' of some edge in \( E_{k+1} \) for \( k\geq 0
\). One puts $V_k=\{1,\ldots,C(k)\}$. The \emph{level} \( k \) is
the subgraph consisting of the vertices in \( V_{k}\cup V_{k+1} \)
and the edges \( E_{k+1} \) between these vertices. Level \( 0 \)
is called the \emph{hat} of the Bratteli-Vershik diagram and it is
uniquely determined by an integer vector
$H(1)=\left(h_1(1),\ldots, h_{C(1)}(1) \right)^T \in \NN ^{C(1)}$,
where each component represents the number of edges joining \(
v_{0} \) and a vertex of \( V_{1} \).
We will assume
$H(1)=(1,\ldots,1)^T$; it is not restrictive for our purpose.

We describe the edge set \( E_{k} \) using a \( V_{k}\times
V_{k-1} \) incidence matrix $M(k)$ for which its \( (i,j)
\)--entry is the number of edges in \( E_{k} \) joining vertex \(
j\in V_{k-1} \) with vertex \( i\in V_{k} \). We also assume
$M(k)>0$ (this is not a restriction for our purpose). 
For $1\leq k \leq l$ one defines 

$$
P(k)=M(k)\cdots M(2)
\hbox{ and } 
P(l,k)=M(l) \cdots M(k+1)
$$ 

with $P(1)=I$ and $P(k,k)=I$ ; where $I$ is the identity map.
Also, put $H(k)=P(k)H(1)=(h_1(k),\ldots,h_{C(k)}(k))^T$.

\subsubsection{Ordered Bratteli-Vershik diagrams}
An \emph{ordered} Bratteli-Vershik diagram is a triple \( B=\left( V,E,\preceq
\right) \) where \( \left( V,E\right)  \) is a Bratteli-Vershik diagram and \( \preceq  \)
 a partial ordering on \( E \) such that : Edges
\( e \) and \( e' \) are comparable if and only if they have the
same end-point.

Let \( k\leq l \) in \( \NN \setminus \left\{ 0\right\}  \) and
let \( E_{k,l} \) be the set of all paths in the graph joining
vertices of \( V_{k-1} \) with vertices of \( V_{l} \). The
partial ordering of \( E \) induces another in \( E_{k,l} \) given
by \( \left( e_{k},\ldots ,e_{l}\right) \preceq \left(
f_{k},\ldots ,f_{l}\right)  \) if and only if there is \( k\leq
i\leq l \) such that \( e_{j}=f_{j} \) for \( i<j\leq l \) and \(
e_{i}\preceq f_{i} \).

Given a strictly increasing sequence of integers \( \left(
m_{n}\right) _{n\geq 0} \) with \( m_{0}=0 \) one defines the
\emph{contraction} of \( B=\left( V,E,\preceq \right)  \) (with
respect to \( \left( m_{n} \right) _{n\geq 0} \)) as $$ \left(
\left( V_{m_{n}}\right) _{n\geq 0},\left(
E_{m_{n}+1,m_{n+1}}\right) _{n\geq 0},\preceq \right) , $$ where
\( \preceq  \) is the order induced in each set of edges \(
E_{m_{n}+1,m_{n+1}} \).
The converse operation is called {\it miscroscoping} (see \cite{HPS} for more details).

Given an ordered Bratteli-Vershik diagram \( B=\left( V,E,\preceq
\right) \) one defines \( X_{B} \) as the set of infinite paths \(
\left( e_{1},e_{2},\cdots \right)  \) starting in \( v_{0} \) such
that for all \( i\geq 1 \) the end-point of \( e_{i}\in E_{i} \)
is the initial-point of \( e_{i+1}\in E_{i+1} \). We topologize \(
X_{B} \) by postulating a basis of open sets, namely the family of
\emph{cylinder sets}
$$
\left[ e_{1},e_{2},\ldots ,e_{k}\right] =\left\{ \left. \left(
x_{1},x_{2},\ldots \right) \in X_{B}\textrm{ }\right| \textrm{
}x_{i}=e_{i},\textrm{ for }1\leq i\leq k\textrm{ }\right\} .$$
 Each \( \left[ e_{1},e_{2},\ldots ,e_{k}\right]  \) is also closed, as is
easily seen, and so we observe that \( X_{B} \) is a compact,
totally disconnected metrizable space.

When there is a unique \( x=\left( x_{1},x_{2},\ldots \right) \in
X_{B} \) such that \( x_{i} \) is maximal for any \( i\geq 1 \)
and a unique \( y=\left( y_{1},y_{2},\ldots \right) \in X_{B} \)
such that \( y_{i} \) is minimal for any \( i\geq 1 \), one says
that \( B=\left( V,E,\preceq \right)  \) is a \emph{properly
ordered} Bratteli diagram. Call these particular points \(
x_{\mathrm{max}} \) and \( x_{\mathrm{min}} \) respectively. In
this case one defines a dynamic \( V_{B} \) over \( X_{B} \)
called the \emph{Vershik map}. The map \( V_{B} \) is defined as
follows: let \( x=\left( x_{1},x_{2},\ldots \right) \in
X_{B}\setminus \left\{ x_{\mathrm{max}}\right\}  \) and let \(
k\geq 1 \) be the smallest integer so that \( x_{k} \) is not a
maximal edge. Let \( y_{k} \) be the successor of \( x_{k} \) and
\( \left( y_{1},\ldots ,y_{k-1}\right)  \) be the unique minimal
path in \( E_{1,k-1} \) connecting \( v_{0} \) with the initial
point of \( y_{k} \). One sets \( V_{B}\left( x\right) =\left(
y_{1},\ldots ,y_{k-1},y_{k},x_{k+1},\ldots \right)  \) and \(
V_{B}\left( x_{\mathrm{max}}\right) =x_{\mathrm{min}} \).

The dynamical system \( \left( X_{B},V_{B}\right)  \) is called
the \emph{Bratteli-Vershik system} generated by \( B=\left(
V,E,\preceq \right)  \). The dynamical system induced by any
contraction of \( B \) is topologically conjugate to \( \left(
X_{B},V_{B}\right)  \). In \cite{HPS} it is proved that any
Cantor minimal system \( \left( X,T\right)  \) is topologically
conjugate to a Bratteli-Vershik system \( \left(
X_{B},V_{B}\right)  \). One says that \( \left( X_{B},V_{B}\right)
\) is a \emph{Bratteli-Vershik representation} of \( \left(
X,T\right)  \). In the sequel we identify $(X,T)$ with one of its
Bratteli-Vershik representations. We always assume the
representations are proper.

A Cantor minimal system is of (topological) finite rank $d\geq 1$
if it admits a Bratteli-Vershik representation such that the
number of vertices per level verify $C(k)\leq d$ for any $k\geq
1$. 
Contracting and microscoping the diagram if needed one can assume (this is done in the
sequel) that $C(k)=d$ for any $k\geq 2$.

A Cantor minimal system is linearly recurrent if it admits a
Bratteli-Vershik representation such that the set $\{M(n);n\geq
1\}$ is finite. Clearly linearly recurrent Cantor minimal systems
are of finite rank (for details about these systems see
\cite{Du1}).

\subsubsection{Associated Kakutani-Rohlin partitions and invariant measures}
Let $(X,T)$ be the minimal Cantor system defined by a properly ordered Bratteli diagrams $B=\left( V,E,\preceq \right)$.

The ordered Bratteli diagram defines for each $n\geq 0$ a
clopen {\it Kakutani-Rohlin} (KR) partition of $X$

$$
\P(n)=\{T^{-j}B_k(n); k \in V_n, \ 0 \leq j < h_k(n) \} \ ,  
$$

with
$
B_k (n) = [e_1 , \dots e_n] ,
$
where $(e_1 , \dots ,e_n)$ is the unique path from $v_0$ to $k$ such that each $e_i$ is minimal for the ordering of $B$.
For each $k \in V_n$ the set $\{T^{-j}B_k(n); 0 \leq j < h_k
(n)
\}$ is the $k$-th tower of $\P(n)$. 
This corresponds to the set of all the paths from $v_0$ to $k\in V_n$ (there are exactly $h_k (n)$ such paths).
The map $\tau_n: X \to V_n$ is
given by $\tau_n(x)=k$ if $x$ belongs to the $k$-th tower of
$\P(n)$. 
Denote by $\T_n$ the $\sigma$-algebra generated by partition the
$\P(n)$; that is, the finite paths joining $v_0$ with any vertex
of $V_n$.

Let $\mu$ be a $T$-invariant measure. It is determined by its
value in $B_k(n)$ for each $n\geq 0$ and $k \in V_n$. Define
$\mu(n)=(\mu_1(n),\ldots,\mu_{C(n)})^T$ with
$\mu_k(n)=\mu(B_k(n))$. A simple computation yields to the
following fundamental relation: 

$$
\mu(n)=M^T(n+1)\mu(n+1)
$$ 

for any $n\geq 1$.

\subsubsection{Return and suffix maps}

Fix $n\in \NN$. The return time of $x$ to $B_{\tau_n(x)}(n)$ is
given by $r_n(x)=\min\{ j\geq 0; T^jx \in B_{\tau_n(x)}(n) \}$.
Define the map $s_n:X\to \NN^{C(n)}$, called {\it suffix map of
order $n$}, by

$$
(s_n(x))_k=\left|
\{e \in E_{n+1}; x_{n+1} \preceq e, x_{n+1}\not = e, k \hbox{ is
the initial vertex of } e\}\right |
$$ 

for all $x=(x_n) \in X$ and $k\in
V_n$. A classical computation gives (see  for example \cite{BDM})
\begin{align}
\label{e:formulareturn}
r_n(x)= s_0(x)+\sum_{k=1}^{n-1}
\langle s_k(x),P(k)H(1)\rangle  \ .
\end{align}

%%%%%%%%%%%%%%%%%%%%%%%%%%%%%%%%%%%%%%%%%%%%%%%%%%%%%%%%%%%%%%%%%%%%%%%%%%%%%
\section{Continuous eigenvalues : general necessary and sufficient conditions}
\label{sec:conti}
%%%%%%%%%%%%%%%%%%%%%%%%%%%%%%%%%%%%%%%%%%%%%%%%%%%%%%%%%%%%%%%%%%%%%%%%%%%%%

Let $(X,T)$ be a Cantor minimal system given by a Bratteli-Vershik
representation $B=\left( V,E,\preceq \right)$. Recall the
associated sequence of matrices is $(M(n);n\geq 1)$, $M(n)>0$ and
$P(n)=M(n)\cdots M(2)$ for $n\geq 2$, and $P(1)=I$. First we
recall a general necessary and sufficient condition to be a
continuous eigenvalue of $(X,T)$ proved in \cite{BDM}.

\begin{prop}
\label{condcont} Let $\lambda=\exp(2i\pi\alpha)$. The following
conditions are equivalent,

\begin{enumerate}
\item
$\lambda$ is a continuous eigenvalue of the Cantor minimal system
$(X,T)$;
\item
$(\lambda^{r_n (x)} ; n\geq 1)$ converges uniformly in $x$, i.e., the
sequence $(\alpha r_n(x) ; n\geq 1)$ converges
$({\rm mod} \ \ZZ )$ uniformly in $x$.
\end{enumerate}
\end{prop}

It follows that,

\begin{coro}
\label{contgen} Let $\lambda \in S^1 = \{ z\in \CC ; |z| = 1 \}$. 
If $\lambda$ is a
continuous eigenvalue of $(X,T)$ then
$$
\lim_{n\to \infty} \lambda^{h_{j_n} (n)} = 1
$$
uniformly in $(j_n; n\in \NN) \in \prod_{n\in \NN} \{1,\dots,C(n)\}$.
\end{coro}

The following theorem states that the necessary condition in
Theorem 1 part (2) in \cite{BDM} is true in general. Denote by
$\vvert \cdot \vvert$ the distance to the nearest integer vector.

\begin{theo}
\label{teo:nec-cont} Let $(X,T)$ be a Cantor minimal system given
by a Bratteli-Vershik representation $B=\left( V,E,\preceq
\right)$. If $\lambda=\exp(2i\pi \alpha)$ is a continuous
eigenvalue of $(X,T)$ then
$$
\displaystyle \sum_{n\ge 1} \vvert \alpha P(n) H(1) \vvert <\infty
\ .
$$
\end{theo}

Before proving Theorem \ref{teo:nec-cont} we introduce an
intermediate statement which gives a more precise interpretation
to this necessary condition.

\begin{lemma}
\label{twopoints} 
Assume coefficients of matrices $(M(n);n\geq 1)$
are strictly bigger than $1$. Let $(j(n); n\in \NN)$ and $(i(n);
n\in \NN)$ be sequences of positive integers such that $j(n+1) -
j(n) \geq 2$ and $i(n) \in \{1, \ldots , C(j(n)) \}$ for all $n\in
\NN$. Then there exist different points $x$ and $y$ in $X$ such
that:
\begin{enumerate}
\item
$\forall \ n\in \NN$, $s_{j(n)}(x)-s_{j(n)}(y)={\rm can}_{i(n)}$ (the
$i(n)$-th canonical vector);
\item
$s_{j} (x) - s_{j} (y) = (0, \dots , 0)^T$ whenever $j\not \in \{
j(n) ; n\in \NN \}$.
\end{enumerate}
\end{lemma}

\begin{proof}
For all $n\in \NN$ we choose $\alpha (n)\in V(n)$.
We defined $x=(e(n))_{n\in \NN}$ and  $y=(f(n))_{n\in \NN}$ and 
 
\end{proof}

\begin{proof}[Proof of Theorem \ref{teo:nec-cont}]
We assume without loss of generality that the coefficients of the matrices
$(M(n);n\geq 1)$ are strictly bigger than $1$ and that
$H(1)=(1,\ldots,1)^T$. Let $\lambda$ be a continuous eigenvalue of
$(X,T)$.

We deduce from Corollary \ref{contgen} that $\vvert \alpha P(n)
H(1) \vvert$ converges to $0$ as $n\to \infty$. Then for all
$n\geq 1$ there exist a real vector $v(n)$ and an integer vector
$w(n)$ such that
$$\displaystyle
\alpha P(n) H(1) = v(n) + w(n) \hbox{ and } v(n) \to_{n\to \infty}
0 \ .$$ Thus it is enough to prove that $\sum_{n\geq 1} \Vert v(n)
\Vert < \infty$.

For $n \geq 1$ let $l(n) \in \{1,...,C(n) \}$ be such that
$$|\langle e_{l(n)}, v(n)\rangle  |=\max_{l \in \{1,...,C(n)\}} |\langle e_l, v(n)\rangle  |$$
where $e_{l}$ is the $l$-th canonical vector of $\RR^{C(n)}$. Let
$$
I^+=\{n \geq 1; \langle e_{l(n)}, v(n)\rangle   \geq 0 \}, \ I^-=\{n \geq 1 ;
\langle e_{l(n)}, v(n)>  <0\}.
$$
To prove $\sum_{n\geq 1} \Vert v(n) \Vert < \infty$ one only needs
to show that
$$
\sum_{n\in I^+} \langle e_{l(n)},v(n)\rangle   < \infty \hbox{ and } -
\sum_{n\in I^-} \langle e_{l(n)},v(n)\rangle   < \infty.
$$
Since the arguments we will use are similar in both cases we only
prove the first one. Moreover, to prove $\sum_{n\in I^+}
\langle e_{l(n)}, v(n)\rangle  < \infty$ we only show $\sum_{n\in I^+\cap 2\NN}
\langle e_{l(n)},v(n)\rangle   < \infty$. 
Analogously, one can prove that
$\sum_{n\in I^+\cap (2\NN+1)} \langle e_{l(n)},v(n)\rangle   < \infty$.

Assume $I^+ \cap 2\NN$ is infinite, if not the result follows
directly. Order its elements $j(0)<j(1)< \ldots <j(n)<\ldots$ and
define $i(n)=l(j(n))$ for $n\in \NN$. 
Let $x,y \in X$ be the
points given by Lemma \ref{twopoints} using these sequences.

Now, from equality \eqref{e:formulareturn} and Proposition
\ref{condcont}, one deduces that

\begin{align*}
&\alpha (r_m (x) - r_m (y)) \\
&= \alpha \sum_{n\in
\{1,...,m-1\}\cap I^+\cap 2\NN} \langle (s_n(x)-s_n(y)),
P(n)H(1)\rangle  + (s_0(x)-s_0(y))\\
&=\alpha \sum_{n\in \{1,...,m-1\}\cap I^+\cap 2\NN} \langle e_{l(n)},
P(n)H(1)\rangle  \\
&=\sum_{n\in \{1,...,m-1\}\cap I^+\cap 2\NN} \langle e_{l(n)},v(n)+w(n)\rangle 
\end{align*}
converges $(~\mod \ZZ)$ when $m \to \infty$. Then $\sum_{n\in
\{1,...,m-1\}\cap I^+\cap 2\NN} \langle e_{l(n)}, v(n)\rangle $ converges
$(~\mod \ZZ)$ when $m\to \infty$. But $\langle e_{l(n)},v(n)\rangle $ tends to
0, hence the series $\sum_{n\in I^+\cap 2\NN} \langle e_{l(n)},v(n)\rangle $
converges.
\end{proof}

The following theorem states that continuous eigenvalues are
always constructed from the subspaces,

$$
\left\{
v \in \RR^{C(m)} ; P(n,m) v\to_{n\to \infty} 0
\right\} , \ m\geq 2,
$$

 defined by the
sequence of incidence matrices $(M(n);n\geq 1)$.
In the sequel we use the norm $||.||$ defined by $||v||= \max_i |v_i|$ for all $v\in \RR^d$.

\begin{theo}
\label{th:condcont} Let $(X,T)$ be a Cantor minimal system given
by a Bratteli-Vershik representation $B=\left( V,E,\preceq
\right)$. Let $\lambda=\exp(2i\pi\alpha)$ be a continuous
eigenvalue of $(X,T)$. Then, there exist $m \in \NN$, $v \in
\RR^{C(m)}$ and $w \in \ZZ^{C(m)}$ such that
$$
\displaystyle \alpha P(m) H(1)=v + w \hbox{ and } P(n,m) v
\rightarrow_{n\to \infty} 0.
$$
\end{theo}

\begin{proof}
We deduce from Corollary \ref{contgen} that $\vvert \alpha P(n)
H(1) \vvert$ converges to $0$ as $n$ tends to $\infty$. Hence, for
every $n \geq 2$ one can write $\alpha P(n)H(1)=v(n)+w(n)$, where
$w (n)$ is an integer vector and $v(n)$ is a real vector with
$\Vert v (n)\Vert\to 0$ as $n\to \infty$. Clearly

\begin{align}
\label{increment}
\alpha P(n+1)H(1)=M(n+1)v(n)+ M(n+1) w(n) = v(n+1)+w(n+1) .
\end{align}
We start proving that there exists $m\geq 1$ such that for all
$n\geq m$ one has $M(n+1)v(n) = v(n+1)$. From Proposition
\ref{condcont} one knows there exists $m\geq 1$ such that for all
$n\geq m$ and all $x\in X$ it holds
$$
\vvert \langle s_n (x),\alpha P(n) H(1)\rangle  \vvert < \frac{1}{4} \hbox{ and
} \Vert v(n) \Vert < \frac{1}{4}  .
$$

Hence, for all $n\geq m$ and all $x\in X$, one gets
\begin{align}
\label{difference} \vvert \langle s_n (x), v(n)\rangle  \vvert < \frac{1}{4} .
\end{align}

Fix $n\geq m$. Consider $x\in B_k (n+1)$ for some $1\leq k \leq
C(n+1)$ and let $0=j_1 < j_2 <\cdots < j_l $ be the collection of
all the integers $0\leq j< h_k (n+1)$ such that $T^{-j}x \in
\cup_{i\in \{1,\ldots,C(n)\}}B_i(n)$. Remark that

\begin{align}
\label{mn}
\Vert s^T_n (T^{-j_l} x) - e_k^T M(n+1) \Vert & = 1 \\
\label{pas} \Vert s_n (T^{-j_{m+1}} x) - s_n (T^{-j_{m}} x) \Vert
& = 1
\end{align}
for all $1\leq m\leq l-1$. Let $1\leq m \leq l-1$ and suppose
$\left| \langle s_n (T^{-j_{m}} x), v(n)\rangle  \right| < 1/4$. Then, from
\eqref{pas},

\begin{align*}
&| \langle s_n (T^{-j_{m+1}} x), v(n)\rangle  | \\ &= | \langle s_n (T^{-j_{m}} x),
v(n)\rangle  + \langle s_n (T^{-j_{m+1}} x), v(n)\rangle  - \langle s_n (T^{-j_{m}} x), v(n)\rangle  |
\\ & < \frac{1}{2} .
\end{align*}

From \eqref{difference} one gets that $ |\langle  s_n (T^{-j_{m+1}} x),
v(n)\rangle  | < \frac{1}{4} $. 
Thus, as $| \langle s_n (x), v(n)\rangle | = 0$, it
follows by induction that $\langle  s_n (T^{-j_{l}} x), v(n)\rangle  <
\frac{1}{4}$. 
Therefore, from \eqref{mn} one deduces that $| \langle e_k,
M(n+1)v(n)\rangle  | < 1/2$. This is true for all $1\leq k \leq C(n+1)$,
then $\Vert M(n+1) v(n) \Vert < 1/2$.

Finally, from \eqref{increment} one deduces that for all $n\geq
m$,
\begin{align}
M(n+1)v(n) = v(n+1) \hbox{ and } M(n+1) w(n) = w(n+1) .
\end{align}

To conclude it is enough to set $v= v(m)$ and $w = w(m)$.
\end{proof}

\begin{lemma} \label{ortho}
Let $(X,T)$ be a Cantor minimal system given by a Bratteli-Vershik
representation $B=\left( V,E,\preceq \right)$. Consider $m \in
\NN$ and $v \in \RR^{C(m)}$ such that $P(n,m)v \to 0$ as $n\to
\infty$. Then $\langle v,\mu(m)\rangle =0$ for any $T$-invariant probability
measure $\mu$.
\end{lemma}
\begin{proof}
From definition one has
\begin{align*}
\langle v,\mu(m)\rangle &=\langle v,P^T(n,m)\mu(n)\rangle =\langle P(n,m)v,\mu(n)\rangle  \\
&\leq \Vert P(n,m)v\Vert \cdot \Vert \mu(n) \Vert
\end{align*}
and the last term converges to $0$ as $n\to\infty$. Thus
$\langle v,\mu(m)\rangle =0$.
\end{proof}

\begin{coro}
\label{description} Let $(X,T)$ be a Cantor minimal system given
by a Bratteli-Vershik representation $B=\left( V,E,\preceq
\right)$ and let $\mu$ be a $T$-invariant probability measure. Let
$\lambda = \exp(2i\pi \alpha )$ be a continuous eigenvalue of
$(X,T)$. Then one of the following conditions holds:
\begin{enumerate}
\item
$\alpha$ is rational with a denominator dividing $\gcd(h_i(m);
1\leq i\leq C(m))$ for some $m \in \NN$.
\item
There exist $m \in \NN$ and an integer vector $w \in  \ZZ^{C(m)}$
such that $\displaystyle\alpha = \langle w, \mu(m)\rangle $.
\end{enumerate}
\end{coro}

\begin{proof}
Let $m$, $v$ and $w$ be as in Theorem \ref{th:condcont}. We recall
that $P(m)H(1)= H(m)$. Assume $v=0$. Then $\alpha H(m)=w$ and thus
$\alpha$ is rational with a denominator dividing $\gcd(h_i(m);
1\leq i\leq C(m))$. Now suppose $v \not = 0$. From Lemma
\ref{ortho}, $\langle v,\mu(m)\rangle =0$. Thus, from $w=\alpha H(m)-v$ and
$\langle \mu(m),H(m)\rangle =1$ one gets $\alpha=\langle w,\mu(m)\rangle $.
\end{proof}

Part (2) of previous lemma left open the question whether any
integer vector $w \in \ZZ^{C(m)}$, for some $m \in \NN$, can
produce a continuous eigenvalue of the system by taking
$\alpha=\langle w,\mu(m)\rangle $. It is enough to consider topological weakly
mixing Cantor minimal systems to see that in some cases not all
integer vectors can produce a continuous eigenvalue. In general,
the set of integer vectors from which one can define continuous
eigenvalues of the system is a discrete group. Normally, very
difficult to describe explicitly. In the next section we give a
slightly more precise description of such group in the finite rank
case.

%%%%%%%%%%%%%%%%%%%%%%%%%%%%%%%%%%%%%%%%%%%%%%%%%%%%%%%%%%
\section{Continuous eigenvalues of finite rank systems}
%%%%%%%%%%%%%%%%%%%%%%%%%%%%%%%%%%%%%%%%%%%%%%%%%%%%%%%%%%

Let $(X,T)$ be Cantor minimal system of finite rank $d$. Fix a
Bratteli-Vershik representation of $(X,T)$ with exactly $d$
vertices per level which sequence of incidence matrices is $(M(n) ; n\geq 1)$.

%%%%%%%%%%%%%%%%%%%%%%%%%%%%%%%%%%%%%%%%%%%%%%%%%%%%%%%%%%
\subsection{Rationally independent continuous eigenvalues
} \label{preuvemultind}
%%%%%%%%%%%%%%%%%%%%%%%%%%%%%%%%%%%%%%%%%%%%%%%%%%%%%%%%%%

Let $E(X,T)$ be the additive group of continuous eigenvalues of
$(X,T)$, that is,
$$E(X,T)=\{\alpha \in \RR ; \exp(2i\pi\alpha) \mbox{ is a
continuous eigenvalue of } (X,T)\} \ .$$

In this section we study the maximal number $\eta(X,T)$ of
rationally independent elements in $E(X,T)$. Remark that $1$ is
always an eigenvalue of $(X,T)$ so $\ZZ \subseteq E(X,T)$. We need
the following lemma whose proof is left to the reader.

\begin{lemma}
\label{dimension} Let $(X,T)$ be a Cantor minimal system of finite
rank $d$. Then, there are at most $d$ ergodic measures
$\mu_1,\ldots,\mu_l$ ($l\leq d$). Moreover, there exits $m \in
\NN$ such that for all $n\geq m$ vectors
$\mu_1(n),\ldots,\mu_l(n)$ are linearly independent.
\end{lemma}

\begin{theo}
\label{th:bound} Let $(X,T)$ be a Cantor minimal system of finite
rank $d$. Let $\mu_1 , \dots , \mu_l$, $l\leq d$, be all its
ergodic measures. Then, $\eta(X,T) \leq d - l + 1$.
\end{theo}

\begin{proof}
Fix a Bratteli-Vershik representation of $(X,T)$ with exactly $d$
vertices per level. Put $\eta=\eta(X,T)$ and assume $\eta >
d-l+1$. Let $\{\alpha_1,\ldots,\alpha_\eta\}$ be a set of
rationally independent elements in $E(X,T)$. From Theorem
\ref{th:condcont}, there exist $m \in \NN$ and vectors $v_i \in
\RR^{d}$, $w_i \in \ZZ^{d}$, for $i\in \{ 1,\ldots,\eta\}$, such
that $\alpha_i H(m)-v_i=w_i$ and $P(n,m)v_i \to_{n\to \infty} 0$.
Consider $m$ so large that Lemma
\ref{dimension} is also verified from such an integer.

From Lemma \ref{ortho} one has that for all $1\leq i \leq \eta$
and all $1\leq j \leq l$, $\langle w_i,\mu_j(m)\rangle =\alpha_i$. Thus,
$\langle w_i,\mu_1(m)-\mu_j(m)\rangle =0$ for $2\leq j \leq l$. Now, from Lemma
\ref{dimension} one deduces that
$\{\mu_1(m)-\mu_2(m),\ldots,\mu_1(m)-\mu_l(m)\}$ generates a
$(l-1)$-dimensional vector space. We conclude that the linear
space generated by $w_1,\ldots ,w_\eta$ is of dimension at most
$d-l+1$. Consequently, there exist integers $\epsilon_1, \ldots ,
\epsilon_{d-l+2}$ with $|\epsilon_1| + \ldots + |\epsilon_{d-l+2}|
\not = 0$ and
$$
\epsilon_1w_1 + \ldots + \epsilon_{d-l+2} w_{d-l+2} = 0 \ .
$$
Thus,
$$
\epsilon_1 \alpha_1 + \ldots + \epsilon_{d-l+2}\alpha_{d-l+2} = 0
$$
which contradicts the fact that $\{\alpha_1, \ldots ,
\alpha_\eta\}$ is a set of rationally independent elements in
$E(X,T)$.
\end{proof}

Remark that from the proof of the theorem, a set of rationally
independent generators of $E(X,T)$ can be determined from a single
level $m$ once we know $\mu(m)$ and good integer vectors. Of
course, this level $m$ can be very large and difficult to get.

\medskip

Put $\eta=\eta(X,T)$. If $\eta=d$ we say that $(X,T)$ is of
maximal type. From Theorem \ref{th:bound} one has that maximal
type systems are uniquely ergodic. Conversely, if $(X,T)$ is
uniquely ergodic, then it has at most $d$ rationally independent
continuous eigenvalues but it is not necessarily of maximal type.
As an example consider a substitution system whose incidence
matrix $M(n)=A$ for all $n\geq 2$ such that $A$ is primitive and
has two real eigenvalues with modulus bigger than one (for details
about Bratteli-Vershik representations of substitution systems see
\cite{DHS}).

Since $1 \in E(X,T)$, one can always produce rationally
independent generators of $E(X,T)$ containing $1$. Observe that
rational eigenvalues are associated to $1$. Fix
$\{1,\alpha_1,\ldots,\alpha_{\eta-1}\}$ a set of rationally
independent generators of $E(X,T)$. Let $\mu$ be an ergodic
measure of $(X,T)$.

From Theorem \ref{th:condcont} there is $m \in \NN$ such that for
all $1 \leq i \leq \eta -1$ there exist a real vector $v_i \in
\RR^d$ and an integer vector $w_i \in \ZZ^d$ satisfying $\alpha_i
H(m)  = v_i +w_i$ and $P(n,m)v_i \to 0$ as $n\to \infty$. From
Lemma \ref{ortho}, each $v_i$ is orthogonal to the linear space
spanned by $\mu(m)$, thus 

\begin{equation}
\label{alphai}
\alpha_i=\langle w_i,\mu (m)\rangle  .
\end{equation}

\begin{prop}
\label{linearspaces} The vectors $\{v_1, \ldots ,v_{\eta-1}\}$ and
the vectors $\{w_1,\ldots,w_{\eta-1},H(m)\}$ are linearly
independent.
\end{prop}

\begin{proof} Suppose
$\sum_{i=1}^{\eta-1} \delta_i w_i=0$ with some $\delta_i \neq 0$.
Since $w_1,\ldots,w_{\eta-1}$ are integer vectors we can assume
the $\delta_i$'s are integer numbers. From $\alpha_i=\langle w_i, \mu(m)
\rangle $ for all $1\leq i \leq \eta-1$ one gets $\sum_{i=1}^{\eta-1}
\delta_i \alpha_i=0$ with some $\delta_i \neq 0$. But
$\alpha_1,\ldots,\alpha_{\eta-1}$ are rationally independent, then
coefficients $\delta_i$'s must be $0$, a contradiction. Then
$w_1,\ldots,w_{\eta-1}$ are linearly independent.

Now, it holds that $H(m)\notin \langle \{w_1,\ldots,w_{\eta-1}\}\rangle $.
Indeed, if $H(m)=\sum_{j=1}^{\eta-1} q_j w_j$, with rational
coefficients, then by taking the inner product with $\mu(m)$ one
gets that $1=\sum_{j=1}^{\eta-1} q_j \alpha_j$. This contradicts
the fact that $1,\alpha_1,\ldots,\alpha_{\eta-1}$ are rationally
independent. One concludes $w_1,\ldots,w_{\eta-1},H(m)$ are
linearly independent.

Therefore, from $\sum_{j=1}^{\eta-1} \lambda_j v_j=0$ one deduces
$(\sum_{j=1}^{\eta-1} \lambda_j \alpha_j) H(m) -
\sum_{j=1}^{\eta-1} \lambda_j w_j=0$ and thus
$\lambda_1=\ldots=\lambda_{\eta-1}=0$.
\end{proof}

Fix an ergodic measure $\mu$ and, for each $n\geq 1$, define $\zeta(\mu,n)$ to be the maximal number
of rationally independent components of $\mu(n)$.

\begin{prop}
\label{muvsalpha} For all $n\geq m$, $\zeta(\mu,n) \geq \eta$. In
particular, if the system is of maximal type, then
$\zeta(\mu,n)=d$ for $n\geq m$.
\end{prop}
\begin{proof}
We give a proof for $n=m$, for a general $n$ it is analogous. Let
$q=(q_1,\ldots,q_d)^T \in \mathbb{Q}^d$ be such that
$\langle q,\mu(m)\rangle =0$. Thus $q$ is not contained in the linear space
$\mathcal{W}$. 
Indeed, if $q=\sum_{i=1}^{\eta-1} r_i w_i + r H(m)$
with $r$ and the $r_i$'s rational numbers, then, taking the inner-product with
$\mu(m)$, one obtains $0=\sum_{i=1}^{\eta-1} r_i \alpha_i + r$
which implies $r_1=\ldots=r_{\eta-1}=r=0$ (recall
$1,\alpha_1,\ldots, \alpha_{\eta-1}$ are rationally independent).

Assume for all subsets $J$ of $\{1,\ldots,d\}$ with cardinality
$\eta$ there is a non zero rational vector $q^J \in \mathbb{Q}^d$
with $q_j^J=0$ for $j\in \{1,\ldots,d\} \setminus J$ such that
$\langle q^J,\mu(m)\rangle =0$. At least $d-\eta+1$ of such vectors must be
linearly independent. To prove this fact consider the family
$J_i=\{i,\ldots,i+\eta-1\}$ for $i\in \{1,\ldots,d-\eta+1\}$ and
the corresponding vectors $q^{J_1},\ldots,q^{J_{d-\eta+1}}$. From
the first part of the proof one concludes that
$H(m),w_1,\ldots,w_{\eta-1},q^{J_1},\ldots,q^{J_{d-\eta+1}}$ are
$d+1$ independent vectors in $\mathbb{R}^d$, which is a
contradiction. Therefore, there is $J \subseteq \{1,\ldots,d\}$
with cardinality $\eta$ such that $\mu_j(m)$, $j\in J$, are
rationally independent components of $\mu(m)$. This gives
$\zeta(\mu,n) \geq \eta$. The maximal type case follows directly.
\end{proof}

Observe that the inequality in the proposition can be strict.

%%%%%%%%%%%%%%%%%%%%%%%%%%%%%%%%%%%%%%%%%%%%%%%%%%%%%%%%%%
\subsection{Dimension group and geometric interpretation of eigenvalues}
%%%%%%%%%%%%%%%%%%%%%%%%%%%%%%%%%%%%%%%%%%%%%%%%%%%%%%%%%%

Observe that it is not enough to have $v\in \RR^d$ with
$P(n,m)v\to 0$ as $n\to \infty$ and $w \in \ZZ^d$ such that
$v+w=\alpha H(m)$ for some $m \geq 1$ to ensure
$\exp(2i\pi\alpha)$ is a continuous eigenvalue of $(X,T)$. In
addition, from Proposition \ref{condcont}, it is also necessary
that the series $\sum_{n\geq m} \langle s_n(x),P(n,m)v\rangle $ converges modulo
$\ZZ$. In the next two propositions we try to give a more precise
statement involving the so called dimension group associated to
the sequence of matrices $(M(n); n\geq 1)$.

Fix $m\geq 1$ and an invariant measure $\mu$. 
Put 

$$
\mathcal{V}(m)
= \langle \{\mu(m)\}\rangle ^{\perp} .
$$

 Let $\mathcal{V}^s(m)$ be the subspace of
$\RR^{d}$ that is asymptotically contracted by $(M(n); n\geq m)$ :

$$
\mathcal{V}^s(m) = \{ v \in \RR^{d} ; P(n,m) v \to 0 \text{ as } n \to \infty\}.
$$

Also distinguish the subspaces of $\mathcal{V}^s(m)$,
$$\mathcal{V}_0(m) = \left\{ v \in \RR^{d} ; \exists n \geq m, P(n,m) v = 0 \right\} =
\bigcup_{n\geq m} \hbox{Ker}(P(n,m))
$$
and
$$\mathcal{V}_1(m) = \left\{ v \in \RR^{d} ; \sum_{n\geq m} \Vert P(n,m) v \Vert <
\infty \right\} \ . $$
Obviously,
$$\mathcal{V}_0(m) \subseteq \mathcal{V}_1(m)  \subseteq \mathcal{V}^s(m) \subseteq \mathcal{V}(m) \ .$$
One has $P(m)\mathcal{V}(1) \subseteq \mathcal{V}(m)$. Equality
holds if the matrices $(M(n); n\geq 1)$ are invertible.

\begin{prop}\label{prop:dimengroup}
There exist $m \geq 1$ and a linear space $\mathcal{V}^{(m)}
\subseteq \mathcal{V}(m)$ such that $P(n,m) : \mathcal{V}^{(m)} \to \mathcal{V}(n) $ is one to one
 for any $n > m$.
\end{prop}
\begin{proof}
Let us choose a subspace $\mathcal{V}^{(1)}$ of $\RR^{d}$ such
that $\mathcal{V}^s(1) = \mathcal{V}^{(1)} \oplus
\mathcal{V}_0(1)$. Let $m \geq 1$ and assume subspaces
$\mathcal{V}^{(n)}$ are defined for all $1 \le n \le m$ such that:
$\mathcal{V}^s{(n)} = \mathcal{V}_0{(n)} \oplus \mathcal{V}^{(n)}$
and $P(n,k) \mathcal{V}^{(k)} \subseteq \mathcal{V}^{(n)}$ for all
$1\leq k < n \leq m$.

Choose a subspace $\mathcal{W}^{(m+1)}$ of $\mathcal{V}^s{(m+1)}$
such that $\mathcal{V}^s{(m+1)}= \mathcal{V}_0{(m+1)} \oplus
M(m+1) \mathcal{V}^{(m)}\oplus \mathcal{W}^{(m+1)}$ and set
$\mathcal{V}^{(m+1)} = M(n+1) \mathcal{V}^{(m)}\oplus
\mathcal{W}^{(m+1)}$. This procedure defines recursively a
sequence of subspaces verifying for all $m\geq 2$ and all $1 \leq
k \leq m$, $\mathcal{V}^s{(m)} = \mathcal{V}_0{(m)} \oplus
\mathcal{V}^{(m)}$ and $P(m+1,k) \mathcal{V}^{(k)} \subseteq
\mathcal{V}^{(m+1)}$.

Since $P(n,m)\mathcal{V}^{(m)} \subseteq \mathcal{V}^{(n)}$ and in
view of the definition of $\mathcal{V}_0{(m)}$, $P(n, m)$ is
injective on $\mathcal{V}^{(m)}$, then the sequence
$(\hbox{dim}(\mathcal{V}^{(m)});n\geq 1)$ is increasing. Since it
is bounded by $d$, there is $m\in \NN$ such that for all $n \geq
m$, $P(n, m)\mathcal{V}^{(m)} = \mathcal{V}^{(n)}$. This concludes
the proof.
\end{proof}

Fix the integer $m$ found in the previous proposition. Notice that if
the matrices $(M(n);n\geq 1)$ are invertible, then one can take $m
= 1$.

Consider the discrete subgroup of $\RR^{d}$
$$
\mathcal{G}(m) = \bigcup_{n\geq m} P(n,m)^{-1} \ZZ^{d} = \left\{z
\in \QQ^{d} ;  \exists n \geq m, P(n,m)z \in \ZZ^{d} \right\}
$$
and the one dimensional subspace $\Delta(m) = \{ t H(m) ; t \in
\RR \} \subseteq \RR^{d}$.

\begin{prop}
\label{group-geo} Let $\lambda=\exp(2i\pi \alpha)$. If $\lambda$
is a continuous eigenvalue of $(X,T)$ then $\alpha H(m) \in
(\mathcal{G}(m) + \mathcal{V}_1{(m)}) \cap \Delta(m) $.
\end{prop}
\begin{proof}
According to Theorem \ref{teo:nec-cont},
$$\displaystyle \sum_{n\ge 1} \vvert \alpha P(n)H(1) \vvert
<\infty \ .$$
From Lemma \ref{ortho}, there exist $m' \geq 1$, an
integer vector $w' \in \ZZ^{d}$ and a real vector $v' \in \RR^{d}$
with
$$
\alpha H(m')= v' + w' \text{ and } \Vert P(n,m')v'\Vert\to 0\text{
as }n\to \infty\ .
$$
One can assume $m' \geq m$.

Since $v' \in \mathcal{V}^s{(m')}$, it splits into $v' = v'_0 +
v'_s$, with $v'_0 \in \mathcal{V}_0{(m')}$ and $v'_s \in
\mathcal{V}^{(m')}$. There is $v \in \mathcal{V}^{(m)}$ such that
$v'_s = P(m',m)v$. Hence,
$$\alpha P(m',m)H(m) =w' + v'_0 + P(m',m)v. $$
Let $n$ be such that $P(n,m')v'_0 = 0$. One has,
$$\alpha P(n,m)H(m) =P(n,m') w' + 0  + P(n,m)v. $$
One deduces that
$$P(n,m) \left(\alpha H(m) - v\right)  \in \ZZ^{d}, $$
which  means that $\alpha H(m) - v \in
\mathcal{G}(m)+\hbox{Ker}(P(n,m))$. To conclude, notice that,
since for $n$ large enough, $\vvert \alpha P(n)H(1) \vvert =
\vvert P(n,m) v \vvert = \Vert  P(n,m) v \Vert $, $v$ must belong
to $\mathcal{V}_1{(m)}$, so that
$$\alpha H(m) \in (\mathcal{V}_1{(m)} + \mathcal{G}(m))\cap \Delta(m).$$
\end{proof}

\begin{remark}
$\mathcal{G}(m)$ is associated to the so called dimension group.
It is classically presented as a quotient
$\mathcal{G}'=\mathcal{H} / \sim $, where
$$
\mathcal{H}= \left\{(z,p) \in \QQ^{d}\times \NN ; \exists n \geq
p, P(n, p)z \in \ZZ^{d} \right\}
$$
and
$$(z,p) \sim (y,q) \Leftrightarrow \exists n \geq p, n\geq  q, P(n,p)z = P(n,q)y. $$

If the matrices $(M(n);n\geq 1)$ are invertible each $g \in
\mathcal{G}'$ can be represented by the unique element $(z,1) \in
\mathcal{H}$ in class $g$. In the general case, some elements of
$\mathcal{G}'$ do not have a representative of this type.
Nevertheless, by previous propositions one can choose an
appropriate $m \geq 1$ and identify each $g \in \mathcal{G}'$ with
a representative of the form $(z,m) \in \mathcal{H}$ if it is not
asymptotically null, and with $(0,m)$ if it is asymptotically
null. The difference here is that two elements of $\mathcal{G}(m)$
may correspond to the same element of $\mathcal{G}'$ if their
images coincide after a while. One has that $\mathcal{G}'$ is
isomorphic to $\mathcal{G}(m)/\approx$ with $z\approx y
\Leftrightarrow \exists n \geq m, P(n,m)z = P(n,m)y$.
\end{remark}

Fix $m\geq 1$ as before and such that
$1,\alpha_1,\ldots,\alpha_{\eta-1}$ is a base of rationally
independent continuous eigenvalues of $(X,T)$ with $\alpha_i
H(m)=v_i + w_i$, $v_i \in \mathcal{V}^s(m)$ and $w_i \in \ZZ^d$.
When $\zeta(\mu,m)=d$, the eigenvalues can be described from
$\mathcal{W}=\langle \{w_1,\ldots,w_{\eta-1},H(m)\}\rangle $.

\begin{prop}
Assume $\zeta(\mu,m)=d$ (in particular if $(X,T)$ is of maximal
type). 
Consider $\alpha=q+\sum_{i=1}^{\eta-1} q_i \alpha_i \in
E(X,T)$ with $q,q_1,\ldots,q_{\eta-1} \in \QQ$. 
Then $q
H(m)+\sum_{i=1}^{\eta-1} q_i w_i $ belongs to $ \mathcal{G}(m)$.
Moreover, if $\alpha$ is a rational continuous eigenvalue then $\alpha H(m)$
belongs to $\mathcal{G}(m)$. 
Conversely, if $q,q_1,\ldots,q_{\eta-1} \in \QQ$
are such that $q H(m)+\sum_{i=1}^{\eta-1} q_i w_i \in
\mathcal{G}(m)$ then $\alpha=q+\sum_{i=1}^{\eta-1} q_i \alpha_i
\in E(X,T)$.
\end{prop}
\begin{proof}
Take $\alpha \in E(X,T)$ as in the statement of the proposition. 
By Proposition \ref{group-geo},
there are $v' \in \mathcal{V}^s(m)$ and $w' \in \mathcal{G}(m)$
such that $\alpha H(m)=v'+w'$. Thus, $v-v'=w'-w$, where
$v=\sum_{i=1}^{\eta-1} q_i v_i$ and $w= q H(n)+\sum_{i=1}^{\eta-1}
q_i w_i$. From $\langle v-v',\mu(m)\rangle =0$ one deduces that
$\langle w-w',\mu(m)\rangle =0$. But $\zeta(\mu,m)=d$, thus $w=w'$ and
consequently $v=v'$. This proves the first result. If $\alpha \in
\QQ$ then $v=0$, which proves the second result.

Now consider $q,q_1,\ldots,q_{\eta-1} \in \QQ$ such that $q
H(m)+\sum_{i=1}^{\eta-1} q_i w_i \in \mathcal{G}(m)$. Put
$v=\sum_{i=1}^{\eta-1} q_i v_i$. The series $\sum_{n\geq m}
\langle s_n(x),P(n,m)v\rangle $ converges uniformly in $x$ modulo $\ZZ$, because
the corresponding series with $v_i$ instead of $v$ does. This
proves $\alpha=q+\sum_{i=1}^{\eta-1} q_i \alpha_i$ belongs to $ E(X,T)$.
\end{proof}

Define the matrix $W=[w_1,\ldots,w_{\eta-1},H(m)]$. 
From \eqref{alphai}, it is direct
that $$W^T \mu(m)=(\alpha_1,\ldots,\alpha_{\eta-1},1)^T \ .$$

\begin{coro}
If $\zeta(\mu,m)=d$ (in particular if $(X,T)$ is of maximal type)
then $E(X,T)$ is isomorphic (as a group) with the discrete
subgroup of $\QQ^d$, $\QQ(X,T)=\{ z \in \QQ^d ; W^Tz \in
\mathcal{G}(m) \}$.
\end{coro}

%%%%%%%%%%%%%%%%%%%%%%%%%%%%%%%%%%%%%%%%%%%%%%%%%%%%%%%%%%%%%%%%%%
\section{Measurable eigenvalues of finite rank systems: a general
necessary condition}\label{section5}
%%%%%%%%%%%%%%%%%%%%%%%%%%%%%%%%%%%%%%%%%%%%%%%%%%%%%%%%%%%%%%%%%%

Let $(X,T)$ be a Cantor minimal system of finite rank $d$ and
$\mu$ a $T$-ergodic measure. By contracting the associated
Bratteli-Vershik diagram one can always assume there exists $I
\subseteq \{ 1, \ldots , d \}$ such that:
\begin{enumerate}
\item
For all $k\in I$, $\liminf_{n\to \infty} \mu\{\tau_n=k\} > 0$;
\item
For all $k\in I^c$, $\sum_{n\geq 1} \mu \{ \tau_n=k \} < \infty$.
\end{enumerate}
A diagram verifying these properties will be called {\it clean}.
From conditions (1) and (2) one deduces that $\tau_n (x)$ belongs
to $I$ from some $n$ for almost all $x\in X$.

Consider a measurable eigenfunction $f:X \to \CC$ of $(X,T,\mu)$
associated to the eigenvalue $\lambda = \exp (2i\pi \alpha )$ with
$|f|=1$, $\mu$-almost surely. For $n\geq 1$ put
$$
f_n = \EE ( f |\T_n ) = \sum_{k=1}^d \sum_{j=0}^{h_k (n)-1} {\bf
1}_{T^{-j} B_k(n)}\frac{1}{\mu_k(n)} \int_{B_k (n)} \lambda^{-j} f
d\mu
$$
and set
$$
\frac{1}{\mu_k (n)} \int_{B_k (n)}  f d\mu = c_k(n)
\lambda^{\rho_k (n)}  ,
$$
with $c_k (n) \geq 0$. If $x$ belongs to the $k$-th tower of level
$n$ one has that $f_n (x) = \lambda^{-r_n (x) + \rho_k(n)}
c_k(n)$. A simple computation yields to the following interesting
relation that we will not exploit in this article: $$c_l(n)
\mu_l(n) \leq \sum_{k=1}^d M_{k,l}(n+1) \mu_k(n+1) c_k(n+1) \, $$
for all $l \in \{1,\ldots,d\}$.

\begin{lemma}
\label{lem:cnk} For any $1 \leq k \leq d$ such that $\liminf_{n\to
\infty} \mu\{\tau_n=k\} > 0$ one has $c_k(n) \to 1$ as $n\to
\infty$.
\end{lemma}
\begin{proof}
By construction, $||f_n||_2^2=\sum_{k=1}^d \mu\{\tau_n=k\}
c_k(n)^2 \to 1$ as $n\to \infty$. Since the $c_k(n)$ are bounded
by one, then $c_k(n) \to 1$ as $n\to \infty$ for each $1 \leq k
\leq d$ such that $\liminf_{n\to \infty} \mu\{\tau_n=k\} > 0$.
\end{proof}

For $n \geq 1$ and $k,l \in \{1 , \ldots , d \}$ define
$S_n(l,k)=\{s_n(x); x\in X, \tau_n(x)=l,\tau_{n+1}(x)=k\}$.

\begin{prop}
Let $(X,T)$ be a Cantor minimal system of finite rank $d$ and
$\mu$ a $T$-ergodic measure. Assume $(X,T)$ is given by a clean
Bratteli-Vershik representation. If $\lambda=\exp(2i\pi\alpha)$ is
an eigenvalue of $(X,T,\mu)$, then for $n\geq 1$ there exist real
numbers $\rho_1(n), \ldots ,\rho_d(n)$ such that the following
series converges,
\begin{align}
\label{necessary-condition} \sum_{n\geq 1} \max_{(l,k)\in J}
\frac{1}{M_{k,l}(n+1)} \sum_{s\in S_{n}(l,k)}
 | 1  - \lambda^{\langle s, H(n)\rangle -\rho_k(n+1) +\rho_l(n) }|^2
\end{align}
where $J  = \{ (l,k) \in \{1,\ldots, d\}^2; \liminf_{n\to \infty}
\mu \{ \tau_n=l, \tau_{n+1}=k\} >0 \}$.
\end{prop}
\begin{proof}

Let $I$ be a subset of $\{ 1, \ldots , d \}$ verifying (1) and
(2) in the definition of a clean Bratteli-Vershik representation.

Let $f:X \to S^1$ be an eigenfunction for the eigenvalue $\lambda
= \exp (2i\pi \alpha )$. 
As above, for $n\geq 1$ and $k\in
\{1,\ldots,d\}$, we set $f_n= \EE ( f |\T_n )$ and $\frac{1}{\mu_k (n)}
\int_{B_k (n)}  f d\mu = c_k(n) \lambda^{\rho_k (n)}$ with $c_k(n)
\geq 0$. From Lemma \ref{lem:cnk}, $c_k(n) \to 1$ as $n\to \infty$
if $k\in I$. Let $l,k \in J$. Observe that $l,k \in I$ too. One
has

\begin{align}
\label{eq:measurebound} 
\frac{b}{M_{k,l}(n+1)} & \leq \frac{ \mu \{
\tau_{n}=l, \tau_{n+1}=k \} }{M_{k,l}(n+1)}  = \frac{M_{k,l}(n+1)
h_l(n) \mu_k(n+1) }{M_{k,l}(n+1)} \\
& = h_l(n) \mu_k(n+1) \ ,
\end{align}

for some $b>0$.
On the other hand, since the sequence $(f_n; n \geq 1)$ is a
martingale, then $\sum_{n\geq 1} ||f_{n+1} - f_n||_2^2$ converges
and

\begin{align*}
&||f_{n+1} - f_n||_2^2 \\
&= \int_X |f_{n+1} - f_n|^2 d\mu \\
& = \int_X \left|c_{\tau_{n+1}(x)}(n+1) \lambda^{-r_{n+1}(x) +
 \rho_{\tau_{n+1}(x)}(n+1)} - c_{\tau_{n}(x)}(n)
\lambda^{-r_{n} (x) + \rho_{\tau_n(x)}(n)}  \right|^2 d\mu \\
&= \int_X c_{\tau_n(x)}(n)\cdot\left |
\frac{c_{\tau_{n+1}(x)}(n+1)}{c_{\tau_{n}(x)}(n)} -
 \lambda^{r_{n+1} (x)-r_{n} (x)- \rho_{\tau_{n+1}(x)} + \rho_{\tau_{n}(x)} }    \right|^2 d\mu
\\
& = \sum_{k=1}^{d} \sum_{l=1}^d h_l(n) \mu_k(n+1) \sum_{s\in
S_{n}(l,k)} c_{l}(n)\cdot \left | \frac{c_k(n+1)}{c_l(n)}  -
\lambda^{\langle s, H(n) \rangle  - \rho_{k}(n+1) + \rho_{l} (n)}
 \right|^2 \\
\end{align*}

Consequently, from the convergence of $c_l(n)$ to $1$ as $n\to
\infty$ for $l\in I$ and \eqref{eq:measurebound} one deduces,

$$ \sum_{n\geq 1} \sum_{(l,k) \in J} \frac{1}{M_{k,l}(n+1) } \sum_{s\in
S_{n}(l,k)}  \left | \frac{c_k(n+1)}{c_l(n)}  - \lambda^{\langle s, H(n)
\rangle  - \rho_{k}(n+1) + \rho_{l} (n)}
 \right|^2
$$
converges. But $| \frac{c_k(n+1)}{c_l(n)}  - \lambda^{\langle s, H(n) \rangle  -
\rho_{k}(n+1) + \rho_{l} (n)}| \geq |\frac{c_k(n+1)}{c_l(n)} -1|$,
then one also gets that
$$ \sum_{n\geq 1} \sum_{(l,k) \in J}  \left| \frac{c_k(n+1)}{c_l(n)} - 1 \right|^2
$$
converges. One concludes that
$$
\sum_{n\geq 1} \sum_{(l,k) \in J} \frac{1}{M_{k,l}(n+1) }
\sum_{s\in S_{n}(l,k)}
 |  1  - \lambda^{\langle s, H(n)\rangle -\rho_k(n+1) +\rho_l(n)}    |^2  .
$$
converges, which gives the result.
\end{proof}

Remark from last theorem that $\frac{1}{M_{k,l}(n+1) } \sum_{s\in
S_{n}(l,k)} \lambda^{\langle s, H(n)\rangle -\rho_k(n+1) +\rho_l(n)}$ converges
to $1$ as $n\to \infty$ for any $(l,k) \in J$. This suggests a
strong condition on the distribution of powers of $\lambda$ in
$S^1$ in relation to the local ordering of the Bratteli-Vershik
representation.

%%%%%%%%%%%%%%%%%%%%%%%%%%%%%%%%%%%%%%%%%%%%%%%%%%%%%%%%%%
\section{Example 1: measurable eigenvalues do not always come from the stable space}
\label{example}
%%%%%%%%%%%%%%%%%%%%%%%%%%%%%%%%%%%%%%%%%%%%%%%%%%%%%%%%%%

In Section \ref{sec:conti} we proved that if
$\lambda=\exp(2i\pi\alpha)$ is a continuous eigenvalue of a 
minimal Cantor system $(X,T)$ given by a Bratteli-Vershik representation
$B=\left( V,E,\preceq \right)$, then for some $m \geq 1$ there
exist $v\in \RR^{C(m)}$ with $P(n,m)v\to 0$ as $n\to \infty$ and
$w\in \ZZ^{C(m)}$ such that $\alpha H(m)= v + w$. In this section
we construct a uniquely ergodic Cantor minimal system of finite
rank $2$ and a measurable eigenvalue $\lambda$ for which this
property is not verified. In particular $\lambda$ will not be a
continuous eigenvalue.

We start by constructing a suitable sequence of matrices
$(M(n);n\geq 1)$. Let $A = \left(
\begin{array}{cc} 1&1\\1&0
\end{array}\right).$ 
Let $\varphi = \frac{1+\sqrt{5}}{2}$ be the Perron eigenvalue of
$A$, $e_u$ an associated eigenvector with positive coordinates and $e_s$ an eigenvector 
of the other eigenvalue $\varphi^{-1}$ such that $\langle e_s,(0,1)^T\rangle >0$ .

As usual $H(1)=(1 \ 1)^T$ and for $n\geq 2$ the matrix
$M(n)=A^{k_n}$ for some integer $k_n \geq 2$ to be defined. Recall
$P(1)=I$ and $P(n) = M(n) \cdots M(2)$ for $n\geq 2$ and $P(n,m) =
M(n) \cdots M(m+1)$ for $1 \leq m \leq n$. We set $K_n =
\sum_{i=2}^n k_i$, so $P(n) = A^{K_n}$. For convenience we set $k_1=K_1=0$.

In $\RR^2$ we distinguish the stable subspace $E^s=\{v\in \RR^2 ; A^n
v\to_{n\to \infty} 0\}$ and the unstable subspace $E^u= \{v\in \RR^2 ;
||A^n v|| \to_{n\to \infty} \infty \}$ vector spaces of $A$.
The vectors $e_u$ and $e_s$ belong respectively to the unstable and the stable spaces.
Moreover,  $\{e_u,e_s\}$ is an orthonormal basis of $\RR^2$.

\begin{lemma}
\label{varstablefaible} Let $(\epsilon_n; n\geq 1)$ and
$(\delta_n; n\geq 1)$ be sequences of real numbers in
$]0,\varphi^{-1}]$ and $v_1 \in ]0 ,  \epsilon_1 [$. There exist a real
number $0<\beta<1$ and a sequence $(k_n; n \geq 1)$ of integers
larger than $2$ such that for all $v>v_1$ and all $n \geq 1$
$$A^{K_n} (\beta e_u + v e_s) = z_n + u_n e_u  + (v_n +
\varphi^{-K_{n}}(v-v_1)) e_s $$ with  $0<v_n < \epsilon_n$, $0<u_n
\leq  \delta_n v_n$ and $z_n \in \ZZ^2$. \end{lemma}

\begin{proof}
First we construct recursively the sequence $(k_n; n \geq 1)$ 
and a sequence $(\alpha_n; n\geq 1)$ such that for all $n\geq 1$

$$ 
z_n = A^{K_n} (\alpha_n e_u + v_1 e_s) -  v_n e_s \in \ZZ^2
$$

for some $0< v_n < \epsilon_n$. We start the recursion with
$\alpha_1 = 0$, $v_1>0$ and $z_1=0$ and put $t_1 = 1$.

Assume the construction is achieved up to $n\geq 1$. Let $k_{min}
\geq 2$ be large enough so that $\varphi^{-k_{min}} v_n
<\epsilon_{n+1}$ and $\varphi^{-k_{min}}<\epsilon_{n+1}$. 
The
direction given by $e_u$ has irrational slope. 
Thus there exist $t_{n+1}>0$,
$0<s_{n+1}<\epsilon_{n+1} - \varphi^{-k_{min}} v_n$ and $\bar
z_{n+1} \in \ZZ^2$ such that $t_{n+1} e_u = \bar z_{n+1} + s_{n+1}
e_s$. Choose $k_{n+1}
> k_{min}$ so that
$\varphi^{-k_{n+1}} t_{n+1}<\min(\frac{\varphi-1}{\varphi}
\delta_n v_n,\varphi^{-1} t_n)$, where $t_n$ is associated to
$z_n$ in previous step. Let
$$v_{n+1} = \varphi^{-k_{n+1}} v_n + s_{n+1} \  \hbox{ and } \ \alpha_{n+1}
= \alpha_n + \varphi^{-K_{n+1}} t_{n+1} \ .$$ One has
\begin{align*}
& A^{K_{n+1}} ( \alpha_{n+1} e_u +  v_1 e_s) - v_{n+1} e_s \\
&= (\varphi^{K_{n+1}} \alpha_n  +   t_{n+1} ) e_u +
(\varphi^{-K_{n+1}} v_1  -\varphi^{-k_{n+1}} v_{n}  - s_{n+1} ) e_s  \\
&= \varphi^{k_{n+1}}( \varphi^{K_{n}} \alpha_n ) e_u
+ \varphi^{-k_{n+1}} (\varphi^{-K_{n}} v_1  - v_{n} )e_s  +(t_{n+1}e_u -s_{n+1}  e_s)  \\
&= A^{k_{n+1}}[ \varphi^{K_{n}} \alpha_n  e_u + (\varphi^{-K_{n}} v_1  - v_{n} )e_s ] + (t_{n+1}e_u -s_{n+1}  e_s)  \\
&= A^{k_{n+1}}z_n + \bar z_{n+1}=z_{n+1} \in \ZZ^2,
\end{align*}
and $0<v_{n+1} = \varphi^{-k_{n+1}} v_n + s_{n+1} <
\epsilon_{n+1}$.

Let $m\geq 0$. It holds,
\begin{align*}
\alpha_{n+m}-\alpha_n&=\sum_{j=n}^{n+m-1}
(\alpha_{j+1}-\alpha_j)=\sum_{j=n}^{n+m-1}
\varphi^{-K_{j+1}}t_{j+1}
\\
&< \varphi^{-K_{n+1}} t_{n+1} \sum_{j=0}^{m-1} \varphi^{-j} <
\varphi^{-K_{n}} \delta_{n} v_{n} < \varphi^{-K_{n}} \delta_{n}
\epsilon_{n} \ .
\end{align*}

Then the sequence $(\alpha_n; n\geq 1)$ converges. 
Put $\beta =
\lim_{n\rightarrow\infty} \alpha_n$. 
It is clear that $0<\beta
<1$ and
$\varphi^{K_n}(\beta-\alpha_n) < \delta_{n} v_{n}$ for $n\geq 1$.
To conclude define $u_n = \varphi^{K_{n}} (\beta - \alpha_n)$ and
observe that for all $v>v_1$
$$ A^{K_n} (\beta e_u + v e_s) = z_n +  \varphi^{K_{n}}
(\beta - \alpha_n)  e_u  + (v_n + \varphi^{-K_{n}}(v-v_1)) e_s $$
where by construction $v_n < \epsilon_n$.
\end{proof}

In previous lemma we gave a procedure to construct one value of
$\beta$. In fact it is possible to construct a whole Cantor set of
such numbers associated to a same sequence $(k_n)_{n \geq 1}$. The
construction in the lemma can be modified as follows: at each step
one finds two different values $t_{n+1}>0$ and $t'_{n+1}>0$ with
$t'_{n+1}>t_{n+1}$ and then we choose $k_{n+1}$ large enough so
that conditions for both values are satisfied. Remark that these
conditions depend on all previous choices for
$t_2,t'_2,\ldots,t_n,t'_n$. Then we can set $\alpha_{n+1} =
\alpha_n + \varphi^{-K_{n+1}} t_{n+1}$ as well as $\alpha'_{n+1} =
\alpha_n + \varphi^{-K_{n+1}} t'_{n+1}$. Since this choice is free
at each step of the recurrence we can construct a Cantor set of
values for $\beta$. It is straightforward that all the values
obtained by this procedure are different. This argument shows that
not all  the possible values of $\beta$ comes from the so called
regular weak stable space, that is $\beta H(1) \notin \ZZ^2 +
E^s$, because the intersection of $\ZZ^2 + E^s$ with $E^u$ is
countable. One has proved,

\begin{prop}
\label{prop:cantorbeta} Let $(\epsilon_n; n\geq 1)$ and
$(\delta_n; n\geq 1)$ be sequences of real numbers in
$]0,\varphi^{-1}]$ and $0< v_1 < \epsilon_1$. There exist a real
number $0<\beta<1$ such that $\beta H(1) \notin \ZZ^2 + E^s$ and a
sequence $(k_n; n \geq 2)$ of integers larger than $2$ such that
for all $v>v_1$ and all $n \geq 2$
$$A^{K_n} (\beta e_u + v e_s) = z_n + u_n e_u  + (v_n +
\varphi^{-K_{n}}(v-v_1)) e_s $$ with  $0<v_n < \epsilon_n$, $0<u_n
\leq  \delta_n v_n$ and $z_n \in \ZZ^2$.
\end{prop}

\begin{coro}
\label{valp} Let $(\epsilon_n; n\geq 1)$ and $(\delta_n; n\geq 1)$
be sequences of real numbers in $]0,\varphi^{-1}]$. There is a
sequence $(k_n;n\geq 2)$ of integers larger than $2$ and a real
number $\alpha>0$ such that for all $n\geq 2$
$$ \alpha P(n) H(1) = w_n + z_n \ ,$$
where $z_n \in \ZZ^2$ and $w_n\in \RR^2$ with $||w_n|| \leq 4
\epsilon_n$.
\end{coro}

\begin{proof}
Let $v_1 < \min \left( \epsilon_1,||\frac{1}{\langle e_u,H(1)\rangle }H(1)-e_u|| \right)$. Let
$0<\beta<1$ be given by Proposition \ref{prop:cantorbeta} with
$v_1$ and the sequences of $epsilon$'s and $delta$'s given there.
We consider the intersection of $\{tH(1); t \in \RR\}$ with
$\{z+\beta e_u + t e_s; t \in \RR\}$ where $z=(1\ 0)^T$. Call it
$\alpha H(1) = z+\beta e_u + v e_s$. By construction one has $v
> v_1$. Then by Proposition \ref{prop:cantorbeta} for $n\geq 2$,
$\alpha P(n) H(1)= P(n)z + z_n + u_n e_u +
(v_n+\varphi^{-K_n}(v-v_1))e_s$. Thus as $\varphi^{-K_n}\leq
\epsilon_n$, $v_n\leq \epsilon_n$ and $u_n\leq \delta_n\epsilon_n$
one concludes.
\end{proof}

Now the matrices $(M(n);n\geq 2)$ have been constructed we will
proceed to give an ordering to the Bratteli diagram induced
by them. 
We introduce the notion of \emph{best ordering} associated
to $(w,h)$ where $w=(w_1,w_2)^T \in \RR^2$ with $w_2\geq 0$ and $h
=(h_1,h_2)^T \in \NN^2$ with strictly positive coordinates such
that the slope $f=|w_1|/|w_2|$ of $\langle \{w\}\rangle ^{\bot}$ is smaller than
$h_2/h_1$. This ordering is described by a word $p=p_1 \ldots p_l 1 $
in $\{1,2\}^*$ of length $l=h_1+h_2$ defined recursively by: set
$p_0=0$ and for $0\leq n \leq l-1$

$$ p_{n+1} =
\left\{
\begin{array}{ll}
1 & \hbox{if } \langle (\sum_{i=1}^{n} e_{p_i}-h),w\rangle   > 0  \\
2 & \hbox{otherwise} \\
\end{array}
\right.
$$
where $e_1,e_2$ are the canonical vectors of $\RR^2$. Let
$w^{\bot}$ be a vector orthogonal to $w$. Consider the line $L=\{h
+ t w^{\bot} ; t \in \RR\}$. Notice that, since $w_2\geq 0$, a
point $y \in \RR^2$ is above this line if and only if $\langle y-h,w\rangle  >
0$. Thus $p_{n+1}=1$ if the integer vector $\sum_{i=1}^{n}
e_{p_i}$ is above the line $L$ and is equal to $2$ otherwise. In
particular, since $\langle h,w\rangle  > 0$  then $p_1=2$. This motivates the
following definition: $K(p) = \inf{\{i \geq 1 : p_i = 1\}}-2$.

\begin{lemma}
\label{geometriebasique} It holds,
\begin{itemize}
\item $K(p) \leq h_2 + \text{sign}({w_1}) f h_1  \leq h_1 (\frac{h_2}{h_1}+ \text{sign}({w_1}) f)$;
\item for all $j \geq K(p)$,
$ \left| \sum_{i = j}^{l} \langle e_{p_i},w\rangle  \right| \leq ||w||$.
\end{itemize}
\end{lemma}
\begin{proof}
The intersection point of $L$ with $\langle \{(0,1)^T\}\rangle $ is
$\frac{\langle h,w\rangle }{w_2}(0,1)^T$. This gives the first inequality. The
second one follows directly when computing the orthogonal
projection of $\sum_{i=j}^{l} e_{p_i}$ over $\{tw; t\in \RR\}$.
\end{proof}

Fix two decreasing sequences of real numbers $(\epsilon_n; n\geq 1)$
and $(\delta_n; n\geq 1)$ with values in $]0,\varphi^{-1}]$ such
that $\delta_n \leq \epsilon_n$ for $n\geq 1$. Let $0< v_1 <
\epsilon_1$. Let $\alpha$ and $(k_n;n \geq 1)$ be as in
Corollary~\ref{valp}. Then $\alpha P(n) H(1)=w_n+z_n$ where $z_n
\in \ZZ^2$ and $w_n\in \RR^2$ for $n\geq 2$. From construction it
follows that $(w_n)_2 >0$.

Let $(X,T)$ be the minimal Cantor system defined from the ordered
Bratteli-Vershik diagram described as follows: (i) the vertex at each level
are labelled by $\{1,2\}$, (ii) the incidence matrices are given
by $M(n) = A^{k_n}$ for $n\geq 2$, and (iii) for $j\in \{1,2\}$
and $n\geq 2$ the order of the $m_j(n)=( M_{j,1}(n),M_{j,2}(n))^T$
edges arriving at vertex $j$ in level $n$ is given by the best
order associated to $(w_n, m_n(j))$ described by the word
$p^{(n,j)}$. Since $p^{(n,j)}_1=2$ and $p^{(n,j)}_{l+1}=1$ this
diagram has unique minimal and maximal points.

Thus for any $x \in X$ and $n \geq 1$ its suffix $s_n(x)$ is given
by,
$$s_n(x) = \sum_{k=o_n(x)+1}^{\langle m_j(n),H(1)\rangle } e_{p^{(n,j)}_k}+e_{1},$$
where $\tau_{n+1}(x)=j$ and $o_n(x)$ is the order of $x_{n+1}$.
Thus, given $i,j\in \{1,2\}$, all vectors of type
$\gamma_n=\sum_{k=o}^{\langle m_j(n),H(1)\rangle } e_{p^{(n,j)}_k}$ with $1\leq
l \leq \langle m_j(n),H(1)\rangle $ are the suffix $s_n(x)$ of some $x \in X$
with $\tau_n(x)=i$ and $\tau_{n+1}(x)=j$ where $i$ is such that
$p^{(n,j)}_o =i$.

Let $\mu$ be the unique invariant measure of $(X,T)$ (it is unique
since $\langle \mu(n),e_s\rangle =0$ for all $n\geq 1$). A direct computation
yields to
$$ \mu\{s_n=\gamma_n \, |\,
\tau_{n}=i,\tau_{n+1}=j\}=\frac{1}{M_{j,i}(n+1)}. $$

\begin{lemma}
There is a positive constant $C$ such that for all $n\geq 1$
$$\mu\{\langle s_n,w_n\rangle  > ||w_n||\}  \leq C \epsilon_n \ .$$
\end{lemma}
\begin{proof}
Let $i,j \in \{1,2\}$. Set $K_j(n)=K(p^{(n,j)})$. From the second
statement of Lemma~\ref{geometriebasique} one gets

\begin{eqnarray*}
\mu\{\langle s_n,w_n\rangle  > ||w_n|| \, | \, \tau_n = i, \tau_{n+1} = j\}
&\leq
& \mu\{ 1\leq o_n <  K_{j}(n) \, | \, \tau_n = i, \tau_{n+1} = j\}\\
&\leq
&\frac{ |\{ 1\leq o < K_j(n) \ ; \ p^{(n,j)}_o = i\}| }{M_{j,i}(n+1)} \\
\end{eqnarray*}

If $\tau_{n}(x)=1$ then necessarily $o_n(x) > K_{j}(n)$, while if
$\tau_{n}(x)=2$ then $|\{ 1\leq o < K_j(n) \ ; \ p^{(n,j)}_o =
2\}|= K_{j}(n)-1$. So in this case
$$\mu\{\langle s_n,w_n\rangle  > ||w_n|| \, | \, \tau_n = i, \tau_{n+1} = j\}
\leq \frac{K_{j}(n)}{M_{j,2}(n+1)} \ .$$

Let $f_n$ be the slope of the orthogonal line defined from $w_n$.
By construction (it is not difficult to verify) one has
$(w_n)_1<0$. Then from Lemma~\ref{geometriebasique} one gets

\begin{align*}
\frac{K_{j}(n)}{M_{j,2}(n+1)} &\leq
\frac{M_{j,1}(n+1)}{M_{j,2}(n+1)}
\left(\frac{M_{j,2}(n+1)}{M_{j,1}(n+1)}- f_n \right) \\ &\leq
\frac{M_{j,1}(n+1)}{M_{j,2}(n+1)} \left(
\left|\frac{M_{j,2}(n+1)}{M_{j,1}(n+1)}- \varphi^{-1} \right| +
\left|\varphi^{-1} - f_n \right| \right) \ .
\end{align*}

Let $w_n=\bar v_n e_s+ u_n e_u$. Recall from construction that
$\bar v_n=v_n + \varphi^{-K_n}(v-v_1)$, $v_n \leq \epsilon_n$ and
$u_n\leq \delta_n v_n$. Also $\varphi^{-k_n}\leq \epsilon_n$.

The slope $f_n$ is given by
$$f_n = \frac{\varphi^{-1} \bar v_n-u_n}{\bar v_n + \varphi^{-1} u_n} = \frac{ \varphi^{-1} -
\frac{u_n}{\bar v_n}}{1 + \varphi^{-1} \frac{u_n}{\bar v_n}}.$$
Thus
$$|f_n-\varphi^{-1}|=
\left | \frac{u_n}{\bar v_n}
\frac{1+\varphi^{-2}}{1+\varphi^{-1}\frac{u_n}{\bar v_n}} \right |
\leq \frac{u_n}{v_n}(1+\varphi^{-2}) \leq (1+\varphi^{-2})
\delta_n \leq (1+\varphi^{-2}) \epsilon_n\ .
$$

On the other hand, $\frac{M_{j,2}(n+1)}{M_{j,1}(n+1)}$ approaches
$\varphi^{-1}$ at speed $\varphi^{-k_{n+1}}\leq \epsilon_{n+1}
\leq \epsilon_n$. Thus
$$\frac{K_{j}(n)}{M_{j,2}(n+1)} \leq C \epsilon_n$$
where $C=2+\varphi^{-2}$.

To conclude, one integrates this uniform bound with respect to $i$
and $j$.
\end{proof}

\begin{remark}
In general, the quantities $\langle s_n(x),w_n\rangle $ are not bounded. But it
is more likely that a point taken at random has $\langle s_n(x),w_n\rangle $ of
order $||w_n||$.
\end{remark}

One assumes $(\epsilon_n;n\geq1)$ is summable ($\sum_{n\geq 1}
\epsilon_n <\infty$).

\begin{theo} The complex number $\exp(2i\pi\alpha)$ is an
eigenvalue of $(X,T,\mu)$ that is not continuous.
\end{theo}
\begin{proof}
The fact that it is not continuous follows directly from
construction and Theorem \ref{th:condcont}.

First we prove the series $ \sum_{n\geq 1} \vvert \langle s_n(x), \alpha
P(n) H(1)\rangle  \vvert$ converges $\mu$-almost surely. Since
$\sum_{n\geq 1} \mu\{ \langle s_n,w_n\rangle  > ||w_n||\} \leq \sum_{n\geq 1}
\epsilon_n < \infty$, then by Borel-Cantelli Lemma one has for
$\mu$-almost $x \in X$
$$\sum_{n\geq 1} 1_{\{ \langle s_n(x),w_n\rangle  > ||w_n|| \}} < \infty.$$

Denote by $N_0(x)$ the first integer such that for all $N >
N_0(x)$, $\langle s_n(x),w_n\rangle  \leq ||w_n||$. Since $N_0$ is almost surely
finite, one has, for $\mu$-almost all $x \in X$ and $N>N_0(x)$,
$$    \sum_{n>N}  \vvert \langle s_n(x), \alpha P(n) H(1)\rangle 
\vvert
 \leq   \sum_{n > N} |\langle s_n(x),w_n\rangle | \leq  \sum_{n >N} ||w_n|| \leq \sum_{n
>N} \epsilon_n \ .$$
Hence the series converges almost surely.

To conclude recall $f(x)=\exp(-2i\pi \sum_{n\geq 1} \langle s_n(x),
\alpha P(n) H(1)\rangle )$ is an eigenfunction of $(X,T)$ associated to
$\exp(2i\pi\alpha)$.
\end{proof}

%%%%%%%%%%%%%%%%%%%%%%%%%%%%%%%%%%%%%%%%%%%%%%%%%%%%%%%%%%
\section{Example 2: continuous and measurable eigenvalues of Toeplitz type systems of finite rank}
%%%%%%%%%%%%%%%%%%%%%%%%%%%%%%%%%%%%%%%%%%%%%%%%%%%%%%%%%%

It is known that any subgroup of $S^1$ can be the set of
measurable eigenvalues of a Toeplitz system (see \cite{DL} or \cite{Dow}).
The main motivation of this section is to show a class of examples
of Toeplitz Cantor minimal systems where the finite rank
assumption restricts the possibilities of measurable eigenvalues.

An ordered Bratteli-Vershik diagram is of {\rm Toeplitz} type if
for all $n\geq 1$ and for all $u,v\in V_{n}$ the number of edges in $E_{n-1}$
finishing at $u$ coincides with the number of edges in $E_{n-1}$
finishing at $v$. Denote this number $q_n$ and set $p_n = q_n
q_{n-1} \cdots q_1$. We say $(q_n;n\geq 1)$ is the characteristic
sequence of the diagram. A Cantor minimal system is said to be of
Toeplitz type if it is given by a Bratteli-Vershik diagram of this
type. This definition is motivated by the characterization of
Toeplitz subshifts in \cite{GJ}. That is, a Bratteli-Vershik
diagram of {\rm Toeplitz} type is a Toeplitz subshift whenever it
is expansive. 
First we prove a known result  for Toeplitz
subshifts.

\begin{theo}
\label{cont-eigen-toeplitz}
Let $(X,T)$ be a Cantor minimal system of Toeplitz type given by a
Bratteli-Vershik system with characteristic sequence $(q_n;n\geq
1)$. Then, $\exp(2i\pi\alpha)$ is a continuous eigenvalue of
$(X,T)$ if and only if $\alpha = \frac{a}{p_n}$ for some $a \in
\ZZ$ and $n\geq 1$.
\end{theo}
\begin{proof}
Let $\exp(2i\pi\alpha)$ be a continuous eigenvalue of $(X,T)$ with
$\alpha \in ]0,1[$. Let $\alpha = \sum_{i\geq 1} \frac{a_i}{p_i}$
with $a_i \in \{0,\ldots,q_i-1\}$ for all $i\geq 1$ be the
expansion of $\alpha$ in base $(p_n; n\geq 1)$.

By Theorem \ref{th:condcont} one has that $\alpha p_n \to 0 \mod
\ZZ$ as $n\to \infty$. This implies that $\sum_{i\geq n+1}
\frac{a_i}{q_{n+1} \cdots p_i} \to_{n\to \infty} 0$. 
Recall $H(1)=(1,\ldots,1)^T$.
From Proposition \ref{condcont} one knows that $\alpha p_n
\langle s_n(x),H(1)\rangle $ converges to $0$ modulo $\ZZ$ and uniformly in $x$.
Let $x_n \in X$ such that $\langle s_n(x_n),H(1)\rangle =b_{n+1}=\lfloor
\frac{q_{n+1}}{2a_{n+1}} \rfloor$. It exists since $\langle s_n(x),H(1)\rangle $
can take any value between $\{0,\ldots,q_{n+1}-1\}$. If $(a_n;n
\geq 1)$ is not ultimately equal to $0$, then
$\lim_{n\to\infty}\alpha p_n \langle s_n(x_n),H(1)\rangle  = \frac{1}{2}$, that
contradicts the fact that it is $0$ modulo $\ZZ$. One concludes
$(a_n;n \geq 1)$ is ultimately equal to $0$ and that $\alpha =
\frac{a}{p_m}$ for some $a\in \NN$ and $m\in \NN$.

Conversely, assume $\alpha = \frac{a}{p_m}$ for some $a\in \NN$
and $m\in \NN$. Then for all $x \in X$ and $n\geq m$ one has
$$\langle s_n(x), \alpha P(n) H(1)\rangle = a \frac{p_n}{p_m} \langle s_n(x),H(1)\rangle 
=a q_{m+1}\cdots q_{n} \langle s_n(x),H(1)\rangle   ,
$$

which belongs to $\ZZ$.
Then
$\sum_{n\geq 1} \langle s_n(x),\alpha P(n)H(1)\rangle $ converges uniformly
modulo $\ZZ$. One concludes by using Proposition \ref{condcont}.
\end{proof}

Let $(X,T)$ be a minimal Cantor system given by a Bratteli-Vershik
diagram of Toeplitz type. The next proposition shows that in the
class of linearly recurrent systems of Toeplitz type, continuous
and measurable eigenvalues coincide.

\begin{theo}
\label{toepLR} Let $(X,T)$ be a
Toeplitz type system with finite rank
 and $\mu$ be the unique $T$-invariant probability
measure. Let $(q_n;n\geq 1)$ be the characteristic sequence of the
associated diagram and suppose it is bounded. 
Then $\exp(2i\pi\alpha)$ is an eigenvalue of
$(X,T,\mu)$ if and only if $\alpha=\frac{a}{p_m}$ for some $a\in
\ZZ$ and $m\in \NN$. In particular, they are all continuous
eigenvalues.
\end{theo}
\begin{proof}
Let
 $\exp(2i\pi\alpha)$ be a measurable eigenvalue with $\alpha
\in ]0,1[$.
and $\alpha = \sum_{i\geq 1} \frac{a_i}{p_i}$
with $a_i \in \{0,\ldots,q_i-1\}$ for all $i\geq 1$ be the
expansion of $\alpha$ in base $(p_n; n\geq 1)$.
From \cite{BDM} one knows that $\langle \alpha
P(n)H(1),e_1\rangle =p_n \alpha \to_{n\to \infty} 0 \mod \ZZ$. This
implies that $\frac{a_n}{q_n}$ goes to zero with $n$ goes to infinity.
The characteristic sequence being bounded
one concludes
$(a_n;n \geq 1)$ is ultimately equal to $0$ and that $\alpha =
\frac{a}{p_m}$ for some $a\in \NN$ and $m\in \NN$.
We conclude using Theorem \ref{cont-eigen-toeplitz}.
\end{proof}

Let $(X,T)$ be a minimal Cantor system of Toeplitz type of finite
rank $d$ and let $\mu$ be a $T$-ergodic probability measure. Let
$(q_n;n\geq 1)$ be the characteristic sequence of the associated
Bratteli-Vershik diagram.

Consider $\lambda=\exp(2i\pi\alpha)$ to be a measurable eigenvalue
of $(X,T,\mu)$ and $f:X \to \CC$ to be an associated eigenfunction
with $|f|=1$, $\mu$-almost surely. One has that
$f_n=\EE_\mu(f|{\T}_n)$ converges $\mu$-almost surely and in
$L^2(X,\B(X),\mu)$ to $f$. 
Following the notations of Section \ref{section5}, we recall

$$
f_n(x)=\frac{\int_{B_k(n)} f d\mu}{\mu_k(n)} \lambda^{-j}=c_k(n)\lambda^{\rho_k (n) -j} 
$$

whenever $x \in T^{-j}B_k(n)$ for some $1\leq k \leq d$
and $0\leq j < h_k(n)$. 
We set ${c'}_k(n) = c_k(n)\lambda^{\rho_k (n)} $.
Remark
$$j=\sum_{i=1}^{n-1} \langle s_i(x),P(i)H(1)\rangle = \sum_{i=1}^{n-1}
p_i \langle s_i(x),H(1)\rangle = \sum_{i=1}^{n-1} p_i \bar s_i(x),$$ where $\bar
s_i(x)=\langle s_i(x),H(1)\rangle $. Since the system is of Toeplitz type one
knows that $0\leq \bar s_i(x) <q_{i+1}$. Given $1\leq i,k \leq d$ and
$n\geq 2$ define $S_{k,i}(n)=\{\bar s_n(x) : x \in X,
\tau_n(x)=k,\tau_{n+1}(x)=i \}$.

Let $I=\{i\in \{1,\ldots,d\};\liminf_{n\to \infty} \mu\{\tau_n=i\}
> 0\}$. Contracting the Bratteli-Vershik diagram given $(X,T)$ if needed we can
assume $\sum_{n\geq 1} \mu \{ \tau_n= i \} < \infty $ for all
$i\in I^c$, that is, the representation of $(X,T)$ can be assume clean.

\begin{theo}
Let $(X,T)$ be a Cantor minimal system of Toeplitz type of finite
rank $d$ and let $\mu$ be a $T$-ergodic probability measure. Then
all measurable eigenvalues of $(X,T,\mu)$ are rational.
\end{theo}
\begin{proof}
Contracting if necessary we can assume $(X,T)$ is represented by a clean
Bratteli-Vershik diagram of Toeplitz type. Let $(q_n;n\geq 1)$ be the
characteristic sequence of the diagram and
$\lambda=\exp(2i\pi\alpha)$ be a measurable eigenvalue of
$(X,T,\mu)$. From martingale theorem $f_n \to f$ as $n\to \infty$
$\mu$-a.e and since the Bratteli-Vershik diagram given $(X,T)$ is
clean then $\mu\{x \in X ; \lim_{n\to \infty}
|{c'}_{\tau_n(x)}(n)|=1\}=1$. Hence, by Egoroff theorem, for $\rho<
\frac{1}{8d^2}$ there is a measurable set $A$ such that $\mu(A)> 1-\rho$
and $(f_n)$ converges to $f$ and $|{c'}_{\tau_n(x)}(n)|$ converges to $1$
as $n\to \infty$ uniformly on $A$.

Let $\epsilon=\frac{1}{8d^4}$. Then for all $n<N$ large enough and
$x \in A$ one has $|f_n(x)-f_N(x)|\leq \epsilon$ and
$|{c'}_{\tau_n(x)}(n)|> 2/3$. 
By using the expression of $f_n$ we recall before, 
one has

$$
\left|
\frac{ {c'}_{\tau_N(x)}(N) }{{c'}_{\tau_n(x)}(n)} 
 -
(\lambda^{p_n})^{\bar s_{n,N}(x)} \right| \leq
\frac{\epsilon}{|{c'}_{\tau_n(x)}(n)|} 
$$ 

for every $x \in A$, where
$\bar s_{n,N}(x)=\sum_{i=n}^{N-1} q_{n+1}\cdots q_i {\bar
s}_i(x)$. Put $Q_{n,N}=q_{n+1}\cdots q_{N}$. Clearly $0 \leq \bar
s_{n,N} < Q_{n,N}$.

Assume $\alpha$ is irrational. For any interval $L \subseteq S^1$,
by the unique ergodicity of the rotation by $\lambda^{p_n}$, one
has
$$d_{n,N}(L)=\frac{1}{Q_{n,N}}|\{0\leq s \leq Q_{n,N}-1 : \lambda^{p_ns} \in L \}|
\to_{N\to \infty} |L|$$ uniformly in $L$ of the same length.

Let $L$ be an interval in $S^1$ such that $|L|=\frac{1}{4d^2}$ and
fix $N$ such that $d_{n,N}(L) > |L|/2$. 
In addition we can assume assume the
interval $L$ is disjoint from the set $\left\{\frac{{c'}_i(N)}{{c'}_j(n)} : 1
\leq i,j \leq d \right\}$ and that the distance between $L$ and
$\left\{\frac{c_i(N)}{c_j(n)} : 1 \leq i,j \leq d \right\}$ is bigger than
$2\epsilon$. Therefore,

\begin{align*}
\mu\{x \in X ; \lambda^{p_n \bar s_{n,N}(x)} \in L \} & =
\sum_{i=1}^d \mu_i(N) p_n Q_{n,N} d_{n,N}(L) \\
& = d_{n,N}(L) > \frac{|L|}{2} > \rho \ .
\end{align*}

This implies that $\mu\{ x\in A ; \lambda^{p_n \bar s_{n,N}(x)}
\in L \}>0$ and thus there is $x \in A$ such that $\lambda^{p_n
\bar s_{n,N}(x)} \in L$ while

$$ 
2\epsilon \leq \left | \frac{ {c'}_{\tau_N(x)}(N)}{{c'}_{\tau_n(x)}(n)} -
(\lambda^{p_n})^{\bar s_{n,N}(x)} \right | \leq
\frac{\epsilon}{|{c'}_{\tau_n(x)}(n)|}\ . $$

Thus, $|{c'}_{\tau_n(x)}(n)| \leq 1/2$, which is a contradiction. One
concludes $\alpha$ is rational.
\end{proof}

Let us now show that the measurable eigenvalues can not be any rational numbers.

\begin{prop}
Let $(X,T)$ be a Cantor minimal system of Toeplitz type of finite
rank $d$ and $\mu$ a $T$-ergodic probability measure. Let
$(q_n;n\geq 1)$ be the characteristic sequence of the associated
diagram. If $\exp(2i\pi (p/q))$, with $(p,q)=1$, is a non
continuous rational eigenvalue of $(X,T,\mu)$ then for all $n$
large enough, $\frac{q}{(q,p_n)}\leq d$.
\end{prop}

\begin{proof}
  Let $\lambda = \exp (2i\pi p/q)$, $p,q \in \NN$, $(p,q)=1$, be a non continuous eigenvalue of
$(X,T,\mu)$. 
Thus $\exp (2i\pi /q)$ is a non continuous eigenvalue of $(X,T,\mu)$ and we can suppose $p=1$. 
From Theorem \ref{cont-eigen-toeplitz} we deduce $\exp (2i\pi (q,p_n)/q)$ is a non  continuous eigenvalue for all $n$ large enough.
Hence we can assume $(q,p_n)=1$ for all $n$ large enough.

Fix $n \geq 1$. Let $\epsilon >0$ be such that $\epsilon  < 1/2qd$. Contracting the diagram if
needed, one can suppose $q/q_n < \epsilon$ for all $n \geq 1$. Recall $r_{n+1}
(x) -r_n (x) = \bar{s}_n (x) p_n$ with $0\leq \bar{s_n} \leq
q_{n+1} -1$. For all $0\leq a \leq q-1$ set $S_n (a) = \{ x\in
X ; \bar{s}_n = a \mod q\}$. 
Let $q_{n+1} = kq +r$ with $0\leq r\leq
q-1$. For all $t\in \{ 1, \ldots ,d\}$ one has that

\begin{align}
\label{minor-cong}
\frac{1}{q}-\epsilon
\leq
\frac{1}{q}-\frac{1}{q_{n+1}}
\leq
\frac{1}{q}-\frac{r}{qq_{n+1}}
\leq
\frac{k}{q_{n+1}}
\leq
\mu \{ S_n (a) |  \tau_{n+1} =t \} 
.
\end{align}

From Theorem 7 in \cite{BDM}, there exist real functions $\rho_n:\{1,\ldots,d\} \to \RR$ such that 
$((1/q) r_n - \rho_n; n\in \NN)$ converges $\mu$-almost everywhere modulo $\ZZ$. Thus,  from Egoroff theorem there exists $A \in \B(X)$ with $\mu(A)\geq 1
-\epsilon$ such that $((1/q)r_n - \rho_n; n\in \NN)$ converges
uniformly on $A$ in $\RR/\ZZ $. Thus $((1/q)\bar{s}_n (x) -
(\rho_{n+1} - \rho_n );n\in \NN)$ converges uniformly to 0 on $A$
in $\RR/\ZZ $.

There exists $t\in \{ 1, \ldots ,d\}$ such that $\mu \{\tau_{n+1}
= t\}\geq 1/d $. Hence, using \eqref{minor-cong}, one obtains $\mu
( S_n (a)  \cap \{ \tau_{n+1} =t\} \cap A) \geq
\left((\frac{1}{q}-\epsilon )/d \right)-\epsilon  \geq \frac{1}{qd}-2\epsilon >
0$ for any $0 \leq a \leq q-1$.

Suppose $q > d$. Then there exist $x\in S_n (a)  \cap \{
\tau_{n+1} =t\} \cap A$ and $y\in S_n (b)  \cap \{ \tau_{n+1} =t\}
\cap A$, with $a\not = b$, and $\tau_n (x) = \tau_n (y)$. Then,
$\frac{p_n}{q} (a-b)$ should go to $0$ in $\RR /\ZZ$ because

$$
\frac{p_n}{q} (a-b)
=
\frac{p_n}{q} (\bar{s_n} (x) - \bar{s_n} (y)) -(\rho_{n+1} (x) - \rho_{n} (x)
)
+ (\rho_{n+1} (y) - \rho_{n} (y) ) \to 0 
$$
in $\RR /\ZZ$. But this is not possible because $1\leq |a-b| \leq q-1$ and $(q,p_n)=1$. Hence
$q\leq d$.
\end{proof}

The last theorem implies that an arbitrary subgroup of $S^1$
cannot be the set of eigenvalues of a Toeplitz minimal system of
finite rank.  Also, it is not difficult to deduce from last theorem  that there is a unique $q \leq d$ with $(q,p_n)=1$ for all 
enough large $n$ such that  all non continuous eigenvalues of the same type are in the subgroup generated by $1/q$.  
Finally observe that from \cite{DM} it follows that Toeplitz type Cantor minimal systems of finite rank that have  non continuous eigenvalues are expansive (thus subshifts).  

In the following example we provide a Toeplitz system
of finite rank 3 where $\lambda=-1$ is a non continuous
eigenvalue.

\medskip

{\bf Example.} Let $(l_n;n\geq 1)$ be a strictly increasing
sequence of integers with $l_1=0$. Put $q_n=3^{l_n}$ for
$n\geq 1$. 
Consider the Toeplitz system $(X,T,\mu )$ of finite rank $3$ given
by the Bratteli-Vershik diagram with characteristic sequence
$(q_n;n\geq 1)$ and such that each tower of level $n$ is built by
concatenating towers of previous level in the following way:

$$
1 \rightarrow (12)^{t_n-3}131, \ 2 \rightarrow 1(12)^{t_n-3}31,
\ 3 \rightarrow (12)^{t_n-3}131,
$$ 

where $q_n=2t_n-3$.
We set $Q_n = q_1 q_2 \cdots q_n$ for all $n$.
Let $n\geq 1$. Define $\rho_1(n)=-\rho_2(n)=-\rho_3(n)=1$ and
$f_n (x) =(-1)^j \rho_{k}(n)$ if $x \in T^{-j}B_k(n)$ for $k\in
\{1,2,3\}$ and $0\leq j < h_k(n)$.
%One can check that $\sum
%p_{n-1} (\mu_1(n)+\mu_2(n)+\mu_3(n))$ converges. 
%Observe that $p_{n-1} (\mu_1(n)+\mu_2(n)+\mu_3(n))\to 0$ as $n\to \infty$.
We set $A_n = \{ x\in X ; f_n (x) \not = f_{n+1} (x) \} = \cup_{1\leq i,j\leq 3} A_n (i,j)$ where

$$
A_n (i,j)=\{x \in X; \tau_n (x) = i, \tau_{n+1} (x) = j, f_n (x) \neq f_{n+1} (x) \} .
$$

Let $x\in A_n (1,1)$. We have $f_n (x) = (-1)^j \rho_1 (n)$ and $f_{n+1} (x) = (-1)^{j+2l3^n } \rho_1 (n+1) =f_n (x)$ for some $l\in \NN$.
Studying the other cases one can check that 

$$
A_n = \{ \tau_{n+1} = 3 \} \cup \{ \tau_{n+1} = 2 , \tau_{n} = 3 \} \cup\left( 
\bigcup_{{0\leq k\leq Q_n-1}\atop{Q_{n+1} -Q_n\leq k\leq Q_{n+1} -1} } T^{-k} B_2 (n+1)  
\right)
$$

and 

$$
\mu (A_n) = \frac{1}{q_{n+2}} + \frac{1}{q_{n+1}} \mu ( \tau_{n+1} = 2) + \frac{2}{q_{n+1}} \mu ( \tau_{n+1} = 2) \leq \frac{4}{q_{n+1}} .
$$

As $\sum \frac{1}{q_{n+1}}$ converges one deduces that $\mu(\limsup_{n\to\infty} A_n)=0$.
Then, $(f_n)$ converges $\mu$-almost everywhere and 
one can check that $f\circ T=-f$ $\mu$-almost everywhere.
Hence, $-1$ is a measurable eigenvalue of
the system. 
Theorem \ref{cont-eigen-toeplitz} implies $-1$ is not a continuous eigenvalue.

\bigskip

{\bf Acknowledgments.} The third author is supported by Nucleus
Millennium Information and Randomness P04-069-F. This project was
also partially supported by the international cooperation program
ECOS-Conicyt C03-E03.

\end{document}